\documentclass[12pt]{amsart}

\usepackage{amsmath}
\usepackage{amssymb}
\usepackage{amsthm}
\usepackage[margin=1in, top=1.5in]{geometry}

\usepackage[slantedGreek]{mathptmx}
\usepackage{url}
\usepackage{bbm}
\usepackage{booktabs}
\usepackage{graphicx}
\usepackage{float}

\newcommand{\grad}{\nabla{}}

\newtheorem{theorem}{Theorem}
\newtheorem{lemma}{Lemma}

\newtheorem*{theorem*}{Theorem}

\renewcommand{\u}{{\bf u}}
\newcommand{\uh}{\hat{u}}

\title[RD models]{
Reaction-diffusion models for morphological patterning of \lowercase{h}ESC\lowercase{s}}

\begin{document}

\author{Prajakta Bedekar\textsuperscript{1}, 
Ilya Timofeyev\textsuperscript{2}, Aryeh Warmflash\textsuperscript{3}, and 
Misha Perepelitsa\textsuperscript{4}}

\begin{abstract}
In this paper we consider mathematical modeling of the dynamics of self-organized patterning of spatially confined human embryonic stem cells
(hESCs) treated with BMP4 (gastruloids) described in recent experimental works \cite{Warmflash, Chhabra}.

In the first part of the paper we use the activator-inhibitor equations of  Gierer and Meinhardt  to identify 3 reaction-diffusion regimes for each of the three morphogenic proteins, BMP4, Wnt and Nodal, based on the characteristic features of the dynamic patterning. We identify appropriate boundary conditions 
which correspond to the experimental setup and
perform numerical simulations of the reaction-diffusion (RD) systems, using the finite element approximation, to confirm that the RD systems in these regimes produce realistic dynamics of the protein concentrations.  

In the second part of the paper we use analytic tools to address the questions of the existence and stability of non-homogeneous steady states for the reaction-diffusion systems of the type  considered in the first part of the paper. We find sufficient conditions on the data of the problem under which the system has  an universal attractor.     
\end{abstract}

\footnotetext[1]{Department of Mathematics, University of Houston, Houston, TX.}
\footnotetext[2]{Department of Mathematics, University of Houston, Houston, TX. Author acknowledge the support from NSF grant DMS-1903270.}
\footnotetext[3]{Laboratory of Systems Stem Cell and Developmental Biology, Department of BioSciences, Rice University, Houston, TX. Author acknowledges support from NSF grant MCB-1553228, Simons Foundation grant 511079, and Welch Foundation grant C-2021.}
\footnotetext[4]{Corresponding Author. Email: mperepel@central.uh.edu. Department of Mathematics, University of Houston, Houston, TX. Author acknowledge the support from NSF grant DMS-1903270.}

\maketitle

\section{Introduction}
The use of reaction-diffusion (RD) system to model morphological patterning of developing organism started with a seminal work of Turing \cite{Turing}, who observed that in some cases the homogeneous distributions of chemical with coupled reaction rates are not stable, and upon small perturbations give rise to  sizable, non-homogeneous patterns of chemicals. Cells at the peaks of these patterns can then differ from the remainder of the cells in their cell fates, growth rates, or morphogenetic movements. That is, the morphological patterning is a consequence prior chemical ``pre-patterning.'' This approach have been extensively developed over the years, see for example  Koch and Meinhards \cite{KM},  Gierer and Meinhardt \cite{GM}, Raspopovic  et al. \cite{Rasp}, Nakamura et al. \cite{Nacamura}. 

In the framework of RD systems, pre-patterning may also occur when the solution of a RD system transitions into a non-homogeneous, stable steady-state, after initial activation, see Gierer and Meinhardt \cite{GM}. This scenario is not uncommon, however is harder to treat analytically, since exact formulas for non-trivial steady states are not available.

Other mechanisms of morphogenesis such as of  mechanical and mechanochemical types had been developed, see for example Murray at al. \cite{MOH},  Murray \cite{Murray}. These models take into account not only chemical but also mechanical properties of  treated cell samples, such as, for example, the motion of chemical in the sample induced by the growth of cells.

In this paper we will discuss  recent experimental findings by Warmflash et al. \cite{Warmflash},  Chhabra et al. \cite{Chhabra} and Heemskerk et al. \cite{warm2} on cell fate differentiation of human embryonic stem cells (hESCs) that do not  fit into the Turing paradigm of morphogenesis, or the other mechanisms mentioned above. In a typical experiment, a spacial  confined colony of cells is treated with BMP4 (bone morphogenetic proteins) which leads to differentiation of cells into the three embryonic germ layers:  endoderm, mesoderm, and ectoderm, surrounded by an outer ring of extraembryonic cells. 

It has been established in earlier works, see for example Arnold et al. \cite{Arnold}, that BMP4 results in expression of Wnt and Nodal proteins that are essential for formations of  germ layers.

The evidence accumulated in \cite{Chhabra} indicates that fate differentiation occurs not after the formation of stable patterns of BMP4/Wnt/Nodal, but concurrently, during the propagation of signaling waves of Wnt and Nodal, unlike the pre-patterning in Turing-type process.
Additionally, the terminal, stable distributions of these proteins do not correspond (decisively) to location of the germ layers. For example, Wnt and Nodal signaling, which synergize to generate mesendodderm, both spread into the middle region where ectoderm forms. We mention in passing, that the distributions of proteins do tend to non-homogeneous, almost radial, steady states.  Moreover, Chhabra et al. \cite{Chhabra} performed a series of experiments ruling out the cell motion and cell growth as effective mechanisms of morphogenesis. The later facts points strongly in favor of a RD system as correct mathematical model.

We summarize below some of the
 characteristic features  for the dynamics of distribution of proteins during the patterning, 
 obtained in  Chhabra et al. \cite{Chhabra} and Heemskerk et al. \cite{warm2}, that we address in this paper. 
\begin{enumerate}
\item Initially high and uniform over the entire domain, BMP4 signaling activity evolves into a region of high activity near the boundary and low activity in the middle and central parts of the domain;
\item boundary of BMP4 initiates waves of Nodal and Wnt that move into the interior towards the center of the domain;
\item propagation of the Nodal wave proceeds independently of BMP4 (and Wnt), after a certain activation period;
\item distributions of BMP4, Wnt and Nodal activities tend to steady states, by the end of the experiment, with the peaks at the boundary, middle of the domain and the center of the domain, respectively.
The difference between the peak value of Wnt and its value at the center smaller than the difference between the peak value and the values near the boundary. The final distributions of all proteins  appear to be radial, non-homogeneous steady states. 
\end{enumerate}

While the understanding of all mechanisms is far from being complete,  an effort was made to generate dynamics consistent with (1)--(4) in the framework of activator-inhibitor RD systems. Tewary et al. \cite{Tewary} developed a RD model model that produces realistic patterns of BMP4, with realistic dependence  on parameters, such as the size of the cell colony. However, due the fact that the reaction part of the model is linear, the model requires selection of {\it matching} boundary and initial conditions, that do not reflect the state of the problem at the beginning of the experiment. For example, it assumed that BMP4 inhibitor does not diffuse from the sample boundary, and, initially, peaks at the center of the sample. Thus,  the model is partially ``exogenous''.  

Chhabra et al. \cite{Chhabra} addressed the mathematical modeling of dynamics described in part (3)   of the above list. In this model, BMP4 acts as an activator for Wnt, which in its turn activates Nodal. The key assumption is a reaction term the equation for Nodal, that incorporates a threshold parameter, depending on the concentration of Nodal, that switches production from being Wnt dependent to auto-catalytic. Other assumptions include structurally different mechanisms of activation/inhibition for Wnt and Nodal. 

The analysis of RD systems in the above mentioned papers  relies on the numerical solutions of the corresponding systems of PDEs.

The purpose of the present paper is to develop closed form (``endogenous'') models based on the classical Gierer-Meinhardt activator-inhibitor system, see equations \eqref{eq:u2}, \eqref{eq:v2} below, that reproduce behaviors (1)--(4).  By the closed form we mean, a solution to an initial-boundary value problem where the initial and boundary conditions reflect the actual experimental setup.

In particular, we assume that all substances can diffuse off the colony edge, and  use ``the Newton's law of cooling'' with appropriate  background values of the substances.

In the first part of the paper, using numerical simulations, we show that  realistic dynamics of the proteins can be obtained by choosing suitable reaction coefficients in the activator-inhibitor system. 
The dynamics of BMP4 is best described by a system with a single, stable node, to which we will refer as type 1 system, see section \ref{sec:model}.  The dynamic of Wnt and Nodal systems fit to the patterns produced by type 2 systems, which have a stable node, a stable focus and a saddle point in the phase plane.  Moreover, we show that  different behaviors  of Wnt and Nodal, as described in  (5), can be attributed to the size of  the reaction coefficients alone. This is due to a general fact that scaling reaction coefficients in a RD systems, which leaves the phase portrait unchanged, while retaining the same diffusion coefficient, results in a different dynamics and, in particular, in  different steady states. Thus, we provide another explanation of phenomenological properties (1)--(4) based on the dynamical  differences structurally similar system of PDEs.  The analysis can be useful in providing an estimates on the ranges of the reaction coefficients that distinguish BMP4, Wnt and Nodal at the level of activator-inhibitor systems.

In the second part of the paper we address the stability property in part (4). The main mathematical difficulty comes from the fact that that the RD models (equations + boundary conditions) in question  do not have, in general, homogeneous steady states. Thus, linearizing equations on a constant state, and solving for eigenvalues and eigenfunctions  does not provide meaningful information, because the constant states are not solutions, the fact that sometimes is overlooked in biological literature. The growth of oscillations should be measured with respect to a proper steady state, which, in this case is non-homogeneous.

This leads us to the following problems that we address in this paper: existence of  steady states and  do exist, and  sufficient conditions for stability.

There are local in time, unique, classical solutions to RD systems in the H\''{o}lder space $C^{2+\alpha,1+\alpha/2},$ as was proved by Ladyzhenskaja et al. \cite{Ladyzh} for a more general system of parabolic equations. Estimates on the $max$ norm of the solution is needed to extend solutions for all times $t>0.$  This was done by Rothe \cite{Rothe} for activator-inhibitor systems with no-slip boundary conditions, which does not apply in our case. Moreover, the estimates obtained in \cite{Rothe} depend on the diffusion coefficients, which greatly complicates the asymptotic analysis. We note here, that the special structure of the activator-inhibitor equations does not allow the application of the theory of invariant regions of Chueh et al. \cite{ChCS}, another well-known technique for the asymptotic analysis, see for example a book of Smoller \cite{Smoller}. 

Our approach is to rely on the maximum principle for the parabolic equations to obtain the uniform in $(x,t)$ estimates of the solution. The application of this method places some restrictions on the coefficients of the RD system.  The key point here is that the bounds are independent of the diffusion coefficient. 

Then, we use the energy-types estimates for $L^2$ norms of the solution  and it time derivative to identify the suitable stability condition that implies the exponential decay of the norm of the time derivative. Further analysis required to bound the spatial gradient of the solution in $L^2$ norm, with an upper bound, independent of time. The  later fact allows us to extract a strongly convergent sequence $\u(x,t_n)$ with $t_n\to \infty,$ whose limiting point is a steady state solution of the activator-inhibitor system. Here, $\u$ is a solution vector $\u{}={}(u,v).$ Finally we show that such steady state, $\u_s,$ is the limit of the $\u(x,t),$ at $t\to \infty.$ The stability condition mentioned above, for given source terms $f(x),$ $g(x),$ size of the domain $\Omega$ and the diffusion coefficient $\mu$ restricts the size of the initial data $\max_{\Omega}|\u_0|.$  This condition defines the basin of attraction for the steady state, i.e., there is a ball $B$ in $L^\infty,$ centered at zero, such the steady state $\u_s$ belongs to $B$ and for any initial data $\u_0\in L^\infty\cap C^{2+\alpha},$ solution  $\u(x,t)$ converges to that steady state.

\begin{figure}[t]
\centering
\includegraphics[scale=0.5]{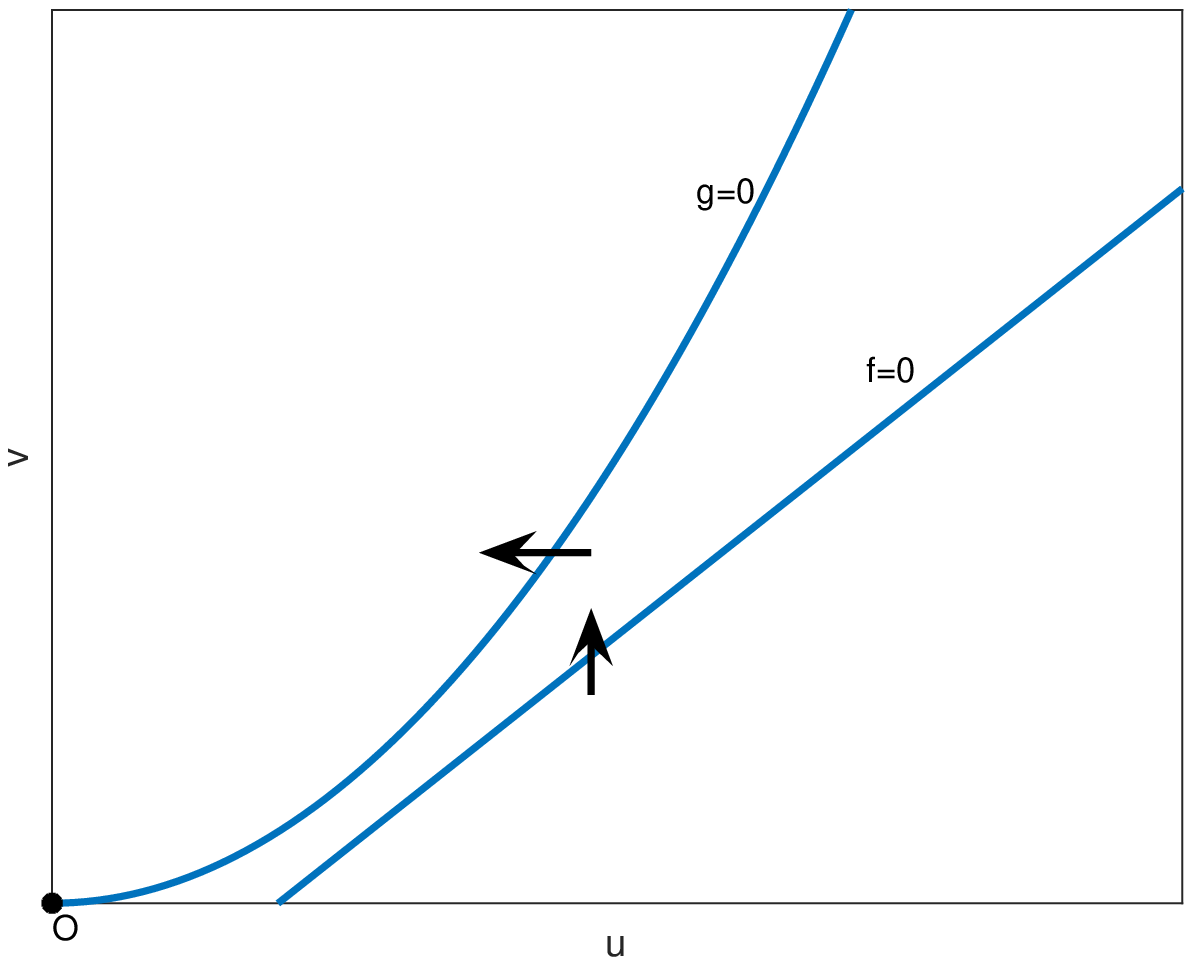}
\includegraphics[scale=0.5]{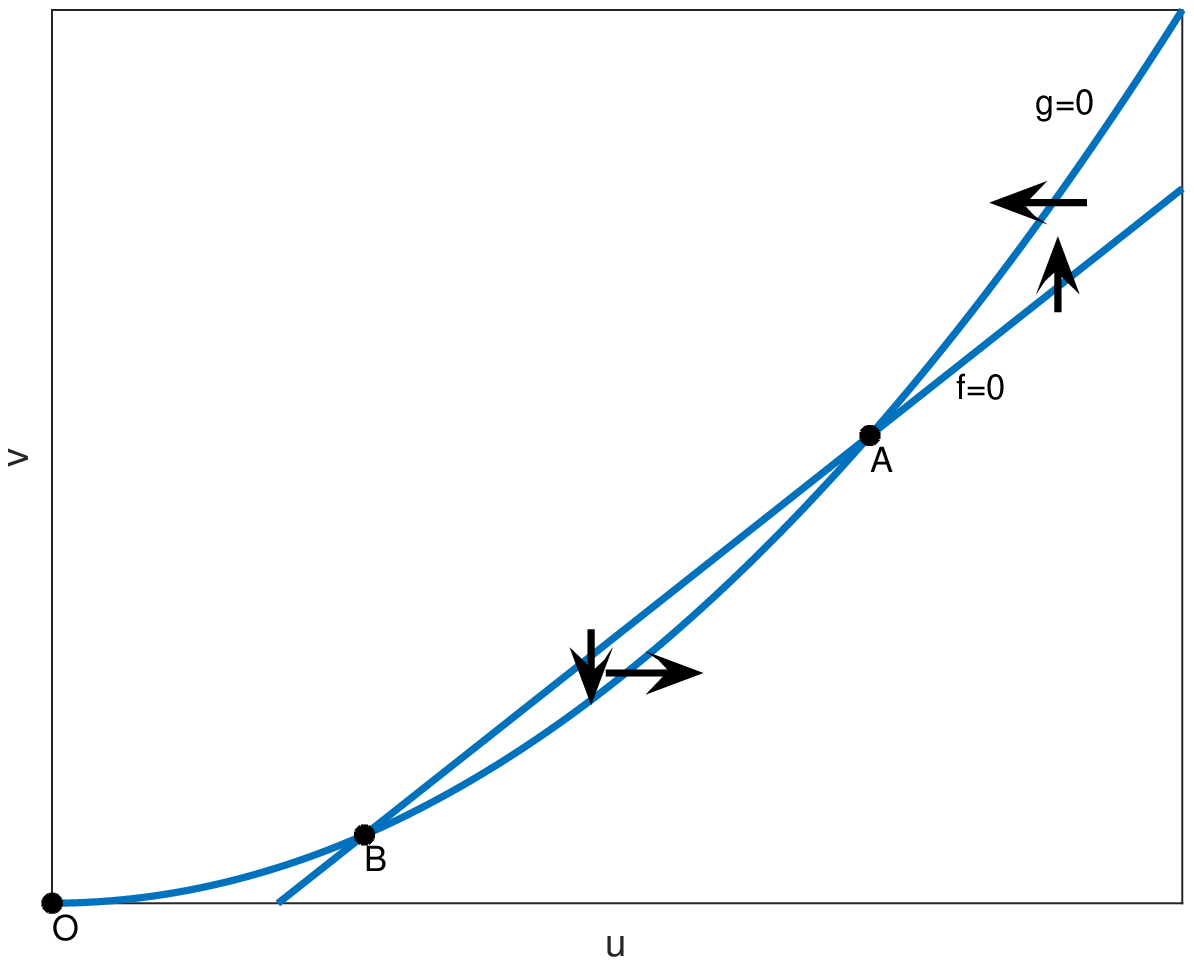}
\caption{Null clines for the reaction dynamics.  Left plot is type 1 reaction with a single fixed point $O$ (stable node). Right plot is type 2 reaction with a stable node $O,$ sable focus $A,$ and saddle $B.$  Arrows show the direction of the flow field.}
\label{fig:nul}
\end{figure}

\section{RD models of activator-inhibitor type}
\label{sec:model}
We consider an activator-inhibitor RD system 
\begin{eqnarray}
\label{eq:u}
\partial_t u{}-{}\mu_u\Delta u{}&=&{}-au {}+{}\frac{bu^2}{1+v},\\
\label{eq:v}
\partial_t v{}-{}\mu_v\Delta v{}&=&{}-cv{}+{}du^2,
\end{eqnarray}
where $\mu_u,$ $\mu_v$ are positive diffusion coefficients, $a,c>0$ are rates of decay and $b,d$ are positive reaction coefficients. This system was introduced by Gierer and Meinhardt \cite{GM}. The inhibitor enters the first equation through the factor $(1+v)^{-1},$ where we added $1$  avoid infinitely high rates when the level of the inhibitor $v$ is small. This is a natural assumption in the experiments described in \cite{Warmflash, Chhabra}. 

Here we denote by $f(u,v)$ and $g(u,v)$ the right-hand side of equations \eqref{eq:u}, \eqref{eq:v}.
There is a single  fixed point (stable node)  $(u_0,v_0){}={}(0,0), $ if $b^2c<4a^2d.$  When $b^2c>4a^2d$ there are three fixed points $(0,0),$ $(u_1,v_1),$ $(u_2,v_2)$ where
\begin{equation}
\label{roots}
u_{1,2}{}={}\frac{cb{}\pm{}\sqrt{(cb)^2{}-{}4a^2cd}}{2da},\quad v_{1,2}{}={}\frac{b}{a}u-1.
\end{equation}
In this case $(0,0)$ is a stable node, as the gradient matrix 
\[
\frac{\partial(f,g)}{\partial(u,v)}{}={}
\left[
\begin{array}{cc}
-a & 0\\
0 & -c
\end{array}
\right].
\]
At other two points, the gradient matrix can be computed to be equal to  
\[
\frac{\partial(f,g)}{\partial(u,v)}{}={}
\left[
\begin{array}{cc}
a & -\frac{a^2}{b}\\
2du & -c
\end{array}
\right],
\]
where $u$ is evaluated at a fixed point. Thus, we obtain the values of the trace and the determinant of the gradient matrix as
\[
{\rm trace}{}={}a-c,\quad  {\rm Det}{}={}\frac{2da^2}{b}u-ac,
\]
These values determine the type of each fixed point. Substituting the values for $u_{1,2}$ from \eqref{roots} we find that ${\rm Det}{}={}\pm\sqrt{(ac)^2-4dca^4b^{-2}}.$
This gives a saddle point, when the value of the determinant is negative. The remaining point is either a stable focus ($a<c$), a center ($a=c$), or an unstable focus ($a>c$).
For the non-linear system \eqref{eq:u}, \eqref{eq:v}, the last two possibilities result in periodic motion. This type of motion is not observed in the experiment discussed in this paper, so we assume that $a<c.$ The null clines  for different values of the decay/reaction coefficients are sketched in Figure \ref{fig:nul}, which we use to distinguish the corresponding RD systems as type 1 and type 2. 

We will use system \eqref{eq:u}, \eqref{eq:v} to model the dynamics of BMP4 and its inhibitor Noggin. To study the signaling waves of Nodal and Wnt we will use a source term $f(x)$ in the activator equation that models the influence of BMP4 on production of Wnt:
\begin{eqnarray}
\label{eq:u2}
\partial_t u{}-{}\mu_u\Delta u{}&=&{}-au {}+{}\frac{bu^2}{1+v}+f(x),\\
\label{eq:v2}
\partial_t v{}-{}\mu_v\Delta v{}&=&{}-cv{}+{}du^2.
\end{eqnarray}
Here $u$ is the concentration of Wnt and $v$ the its inhibitor DKK.
In the numerical simulations we will assume that $f(x)$ is concentrated near the boundary of the domain which reflects the experimentally observed distribution of BMP. The activation of Nodal occurs through ${\rm BMP}\to {\rm Wnt}\to {\rm Nodal}$ signaling pathway. We will model this by a simplifying to  ${\rm BMP}\to{\rm Nodal}$ signaling, and using equations \eqref{eq:u2}, \eqref{eq:v2}, with a different set of reaction coefficients, to model the dynamics of Nodal ant its inhibitor Lefty.

As we will show in by numerical simulation, the system of BMP4 and its inhibitor has a better fit into the reaction system of type 1, see Figure \ref{fig:nul}, while the systems for Wnt and Nodal are better described by type 2 dynamics.


The initial data correspond to the high initial concentration of BMP4 and low (zero) concentrations of other chemicals, in accordance with experiments described in \cite{Warmflash}.

To complete the model, we need to postulate boundary conditions for chemical concentrations.
It should be emphasized that the boundary conditions are the integral part of the solution, that 
plays an important part in the way the dynamics proceeds. According the experimental set up the chemicals can diffuse from the domain of  a cell sample, the intensity of this ``leaking'' being proportional to the difference between the boundary concentration of the chemical and the ``background'' concentration. For BMP4 it is reasonable to take the background concentration at the fixed level $\bar{u}$ that equals to the initial concentration of BMP4. 
For other chemicals in question, the background concentration is zero. This is so-called ``Newton's law of cooling.'' It is expressed as
\[
\frac{\partial u}{\partial n}{}={}h_u(\bar{u}-u),\quad h_u>0,
\]
where $n$ is the external, unit normal vector at $\partial \Omega.$ Similarly, for the inhibitor,
\[
\frac{\partial v}{\partial n}{}={}-h_v v,\quad h_v>0.
\]

We note that an earlier works on mathematical modeling of hESC development such as \cite{Tewary},  used {\it ad hoc} boundary conditions, not consistent with the experimental setting. An alternative way to deal with the boundary, is to embed the reaction domain into a larger domain where chemicals are only diffused, and postulate, for example, no-flux boundary conditions on the larger domain, as was done by in \cite{Chhabra}. The reason being that, the precise form of the boundary conditions on 
the large domain should have minimal effect on the domain where the reaction takes place.

\section{Numerical simulations}
Simulations are performed for RD systems written in scaled variables, using the values of coefficients of the magnitude typically occurred in experimental studies, see Appendix for details. In particular, the computational domain is a disk of radius $1,$ and time $t$ is measured in days. 
The values of parameters for system \eqref{eq:u}, \eqref{eq:v} for BMP4/Noggin, and system \eqref{eq:u2}, \eqref{eq:v2} for Wnt/DKK and Nodal/Lefy, used in the simulations, are listed in table \ref{tab:1}.
\begin{table}
\begin{tabular}{@{} lccc  @{}}    \toprule
& bmp4/noggin & wnt/dkk & nodal/lefty   \\\midrule
$a$ & 77.76  & 77.76  & 31.104    \\ 
$b$ & 77.76 & 194.4 & 77.76 \\ 
 $c$ & 77.76 & 194.4 & 77.76 \\
 $d$ & 77.76 &  97.2   &  38.88     \\
 $\mu_u$ &3.8 & 3.8 & 3.8 \\
 $\mu_v$  &19 & 19  &  19 \\ 
 $h_u$     &172.8& 172.8&172.8\\
 $h_v$     &172.8&172.8&172.8\\  \bottomrule
 \\
\end{tabular}
\caption{Dimensionless reaction and diffusion parameters for the activator-inhibitor RD systems \eqref{eq:u}, \eqref{eq:v} and \eqref{eq:u2}, \eqref{eq:v2}, used in numerical simulations.}
\label{tab:1}
\end{table}
 The parameters are selected is such a way that RD system for BMP4 is of type 1, and the systems for Wnt and Nodal are of type 2. Reaction coefficients for Nodal system differ by a factor of $0.4$ from the corresponding coefficients for Wnt system, which means that the phase portraits for the reaction dynamics are identical in both  cases.
 
The source term $f(x)$ in \eqref{eq:u2} is set to be supported near the boundary of the disk:
\[
f(x) = \left\{
\begin{array}{ll}
670 & |x|>0.85,\\
0 & |x|<0.85.
\end{array}
\right.
\]
Finally, the background state $\bar{u}$ is set to 3 for BMP4, as well as the initial values for BMP4. All other chemicals have zero initial values, and zero background states. 

We will show below that behaviors 1.-4., listed in the Introduction, are captured by the RD models described here. The exact timing of different phenomena described below does not necessarily correspond to experimentally observed values. That would require more precise estimation of the parameters of the model. Our main goal is to establish that qualitatively correct behavior is produced by the model.
\subsection{Terminal concentrations of proteins}
The numerical simulations show that concentrations of all three proteins approach a radial, steady state profiles  by $t=1$ day.  Figure \ref{fig:2d3d} shows 2d and 3d plots of  concentrations of the proteins at $t=3$ days.
BMP4 is concentrated at  the boundary of the domain. Wnt peaks in the middle section, but takes comparable values at the center, and Nodal peaks at the center but somewhat extends to the middle section of the domain.

The appearance of steady states was identified when the change between the successive iterations of the numerical solution became less than $10^{-6}$ units, over a period of time of 1 day.

To illustrate the difference between three different diffusion-reaction regimes we map the radial cut of each protein and its inhibitor in the phase plane on Figure \ref{fig:phase_plane}. Point $B$ indicates the values at the boundary of the disk and point $C$ represents the values at the center.  The plots also show the stable fixed point, $S,$ for the corresponding reaction dynamics.
 BMP4 starts at non-zero value at the boundary due to the influence from non-zero background state $\bar{b}$ and then moves towards zero, according to the reaction, as one moves to the center of the disk. Point $S$ is a stable focus for Wnt/DKK and Nodal/Lefty pairs, however the reaction coefficients are stronger in Wnt system than in Nodal system. This results in the radial profile of Wnt/DKK  being ``bent'' in the direction of the reaction. This property results in maximum of Wnt to be located in the middle of the disk, whereas Nodal has maximum at the center.

\subsection{Signaling waves of Wnt and Nodal}
Figure \ref{fig:T_cuts} shows the evolution of concentrations of Wnt and Nodal as functions of time.
In both cases, the concentration first increase near the boundary, where $f(x)$ is supported and then they move toward the center. This behavior corresponds to the ``signaling waves'' of the proteins described in Introduction.

\subsection{Effect of inhibition of BMP4}
Figure \ref{fig:t_cuts} shows the effect of inhibition of BMP4 ($f(x)=0$) at time indicated by the variable $tcut$ on the shape of the final concentration at time $t=3$ days. 

The simulations show that there is a critical time $t_0$ with the property that if BMP4 is inhibited prior to $t_0$ the system converges to zero steady-state, but when BMP4 is inhibited after $t_0$ the system proceeds autonomously to a non-homogeneous, non-zero steady state. For Wnt dynamics $t_0$ is estimated to be between $0.001$ and $0.005$ days, and for Nodal, between $0.005$ and $0.01$ days, the difference is due to the difference in the magnitude of the reaction coefficients.

\subsection{Dependence on parameters}
\label{sec:3.4}
For small variations of parameters given in Table \ref{tab:1} the numerical simulation produce qualitatively similar results, indicating that the problems are stable. That is, the terminal steady state concentrations are stable. This property is lost when the larger variations. We performed the numerical simulation
of BMP/Noggin dynamics with large gap in diffusion coefficients, by selecting $\mu_u = 1\,\mu m^2/sec$ and $\mu_v = 55\,\mu m^2/sec,$
(instead of $\mu_u=11\,\mu m^2/sec,$ and $\mu_v = 55\,\mu m^2/sec$ used previously), while keeping all other parameters. 

Figure \ref{fig:t2d} shows 2d and 3d plots of  non-radial profile of BMP4 at time $t=3$ days. The non-radial perturbations start to develop from a radially symmetric profile at the time about $t=0.2$ days. Note that,
due to the radial symmetry of equations, the problems has a unique, classical, radially symmetric solution if the initial data have this property, but the numerical solution deviate from it significantly.

A possible explanation of this phenomenon is that the problem has an unstable radially symmetric steady state to which the system moves from its initial values. Small deviations from radially symmetry due the numerical approximation lead to the growth of perturbations shown on the figure. That is, this is the case of Turing instability.

Interestingly, the instabilities appear to be restricted to the boundary, and further simulations (not shown here) produce a different number of peaks, with further variations in the diffusion coefficients. This non-homogeneous profile might, in principle, be associated with the formation of the outer ring of germs in a cell colony. Further investigation of a coupled BMP4-Wnt-Nodal system is needed to clarify if this behavior bears some significance in actual biological processes.

\begin{figure}[h]
\centering
\includegraphics[scale=0.55]{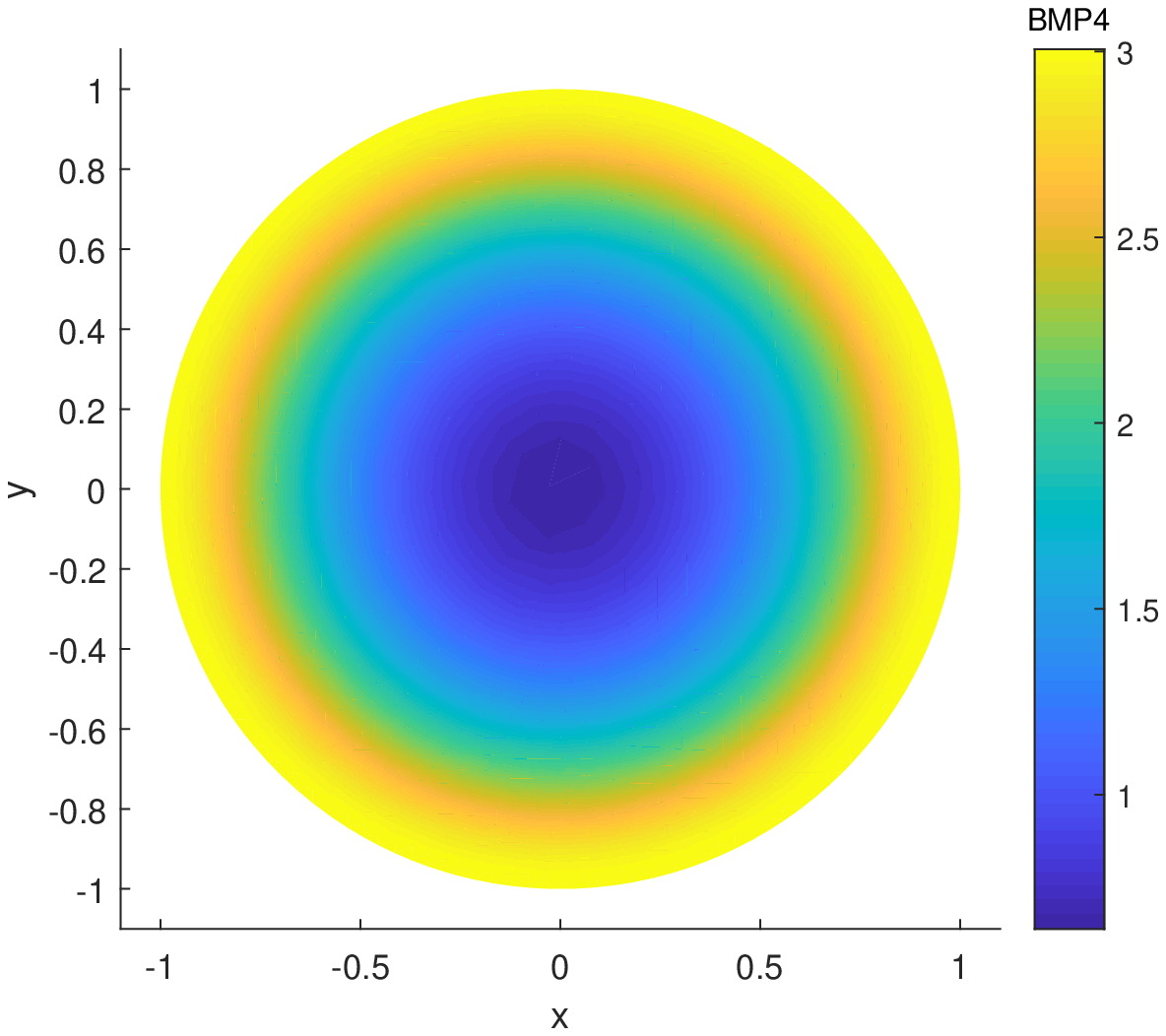}
\includegraphics[scale=0.55]{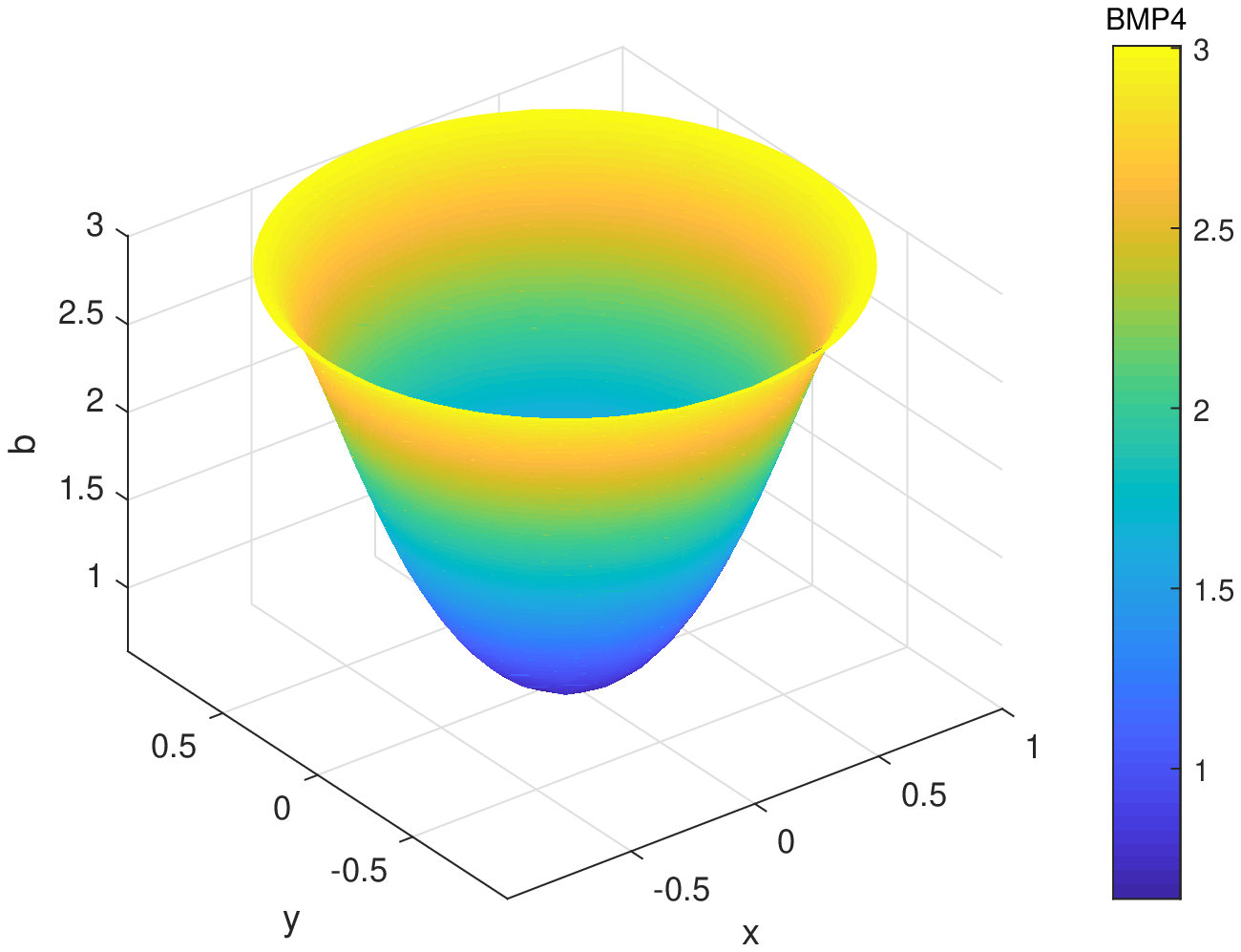}\\
\includegraphics[scale=0.55]{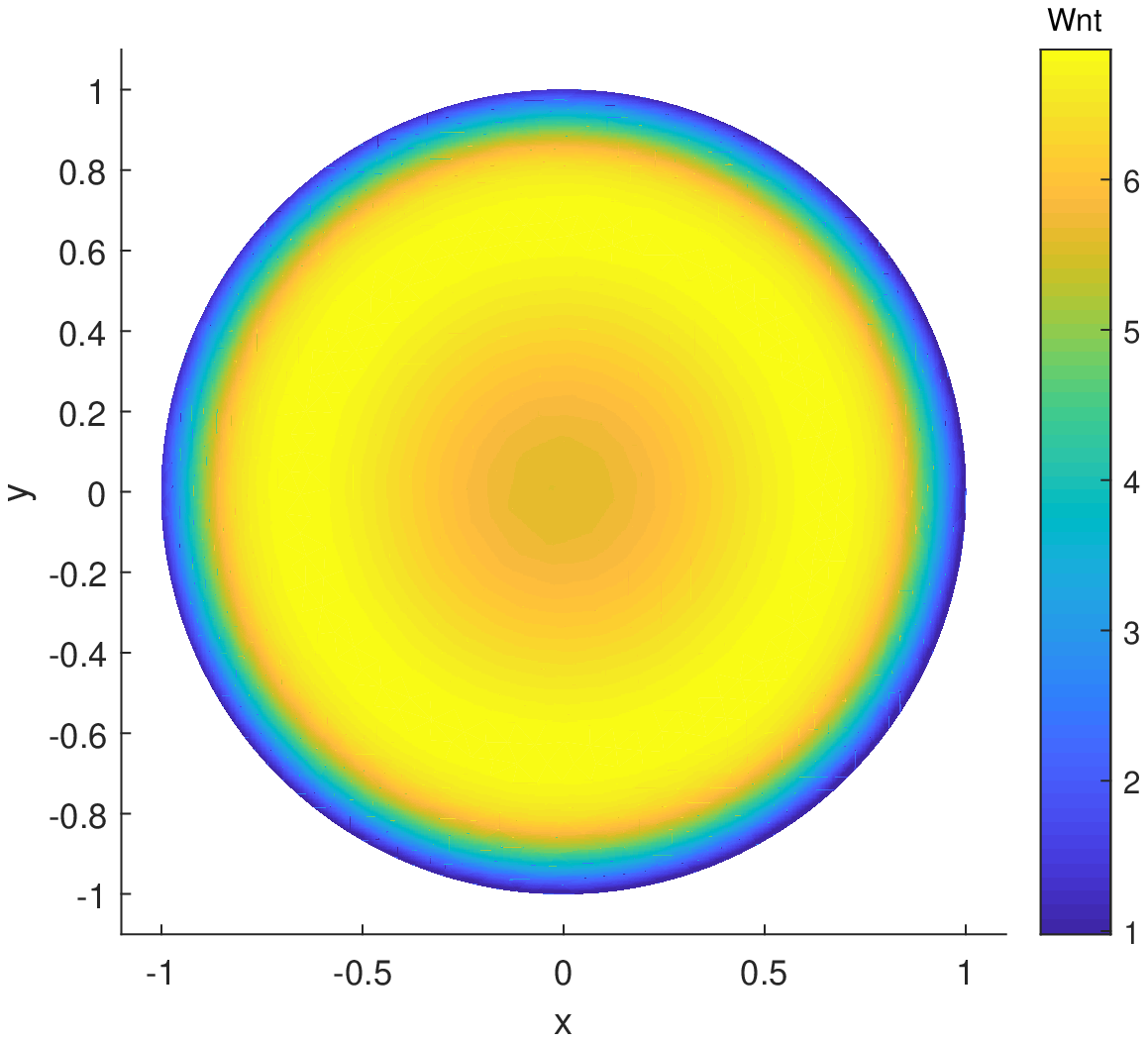}
\includegraphics[scale=0.55]{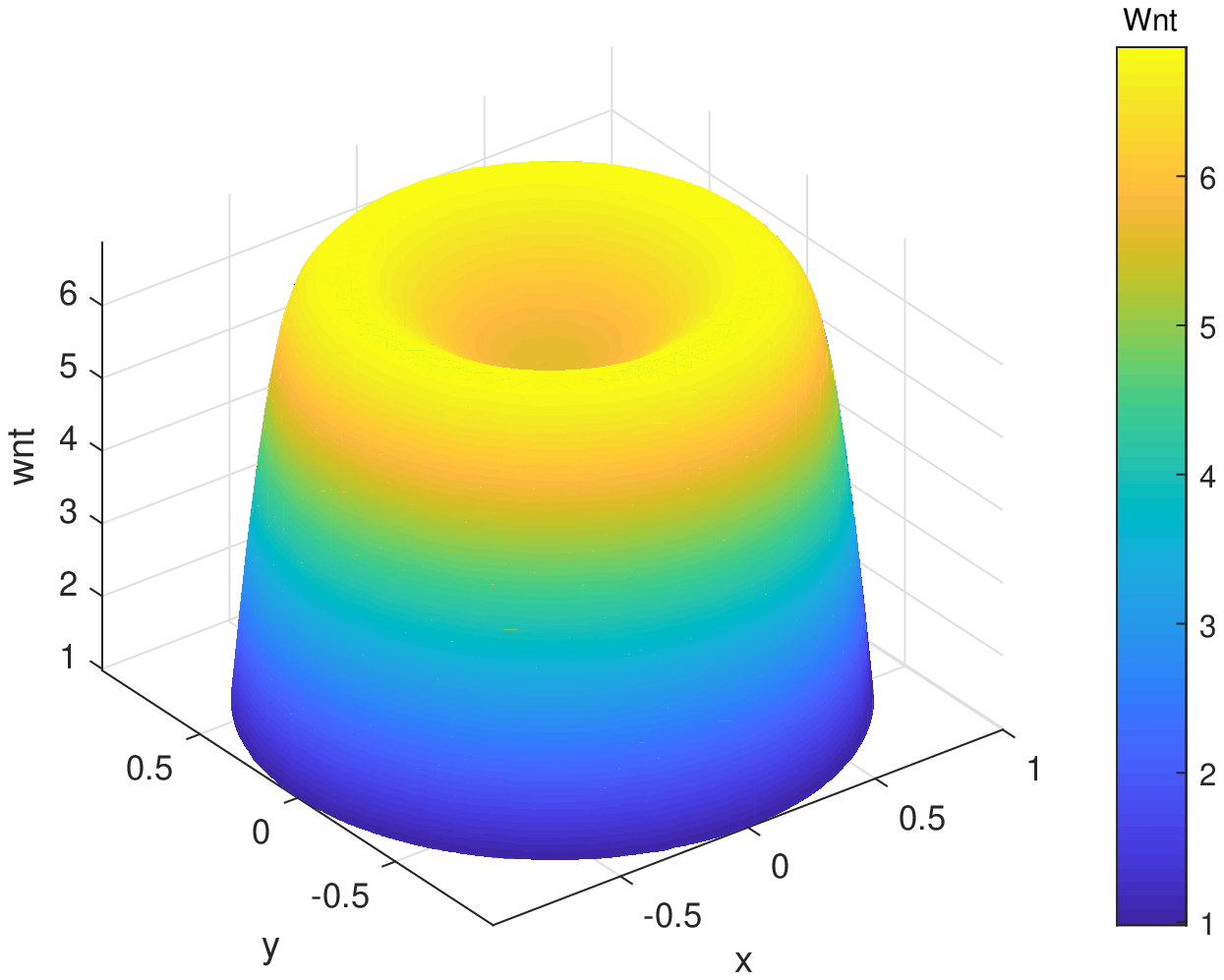}\\
\includegraphics[scale=0.55]{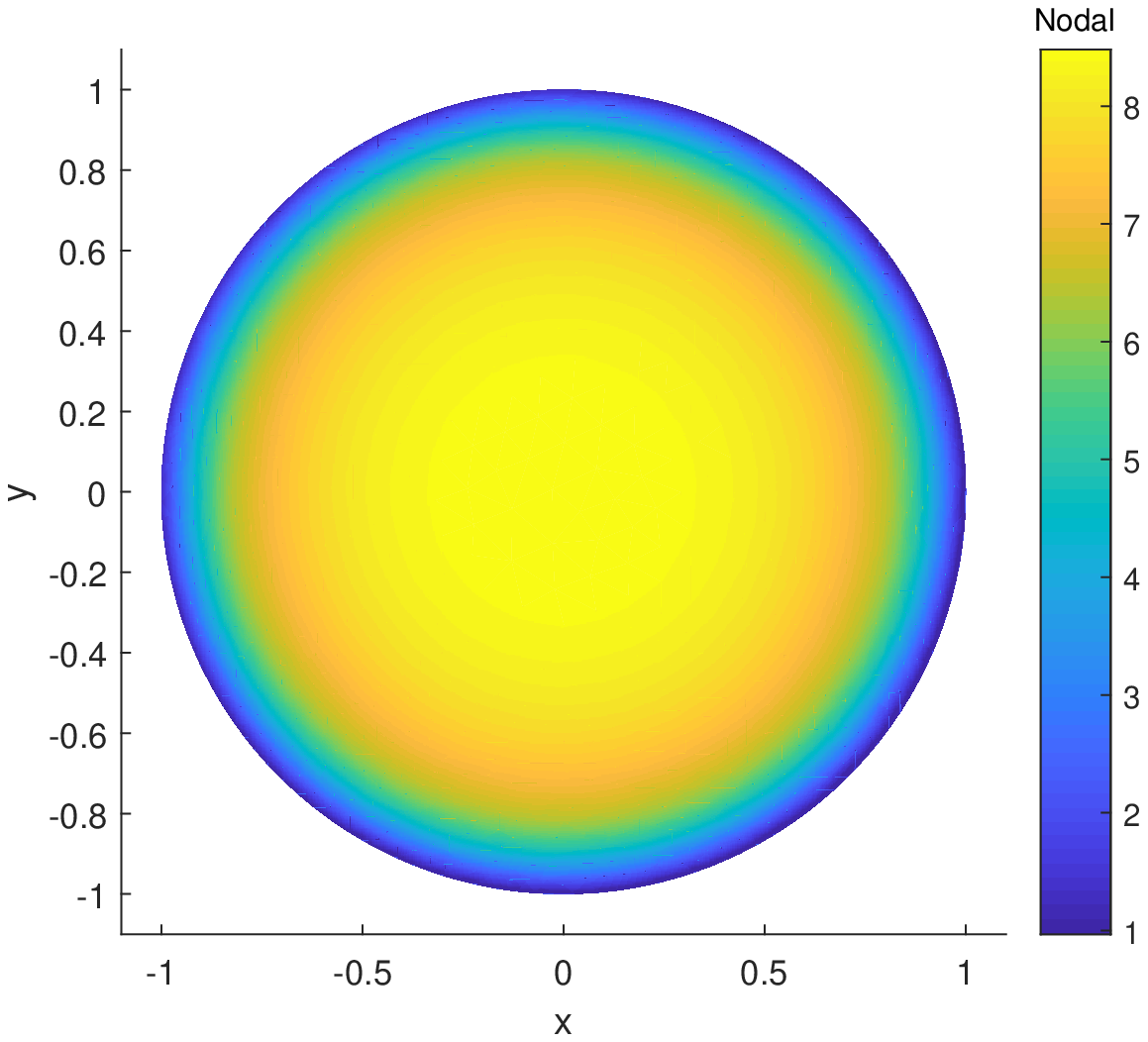}
\includegraphics[scale=0.55]{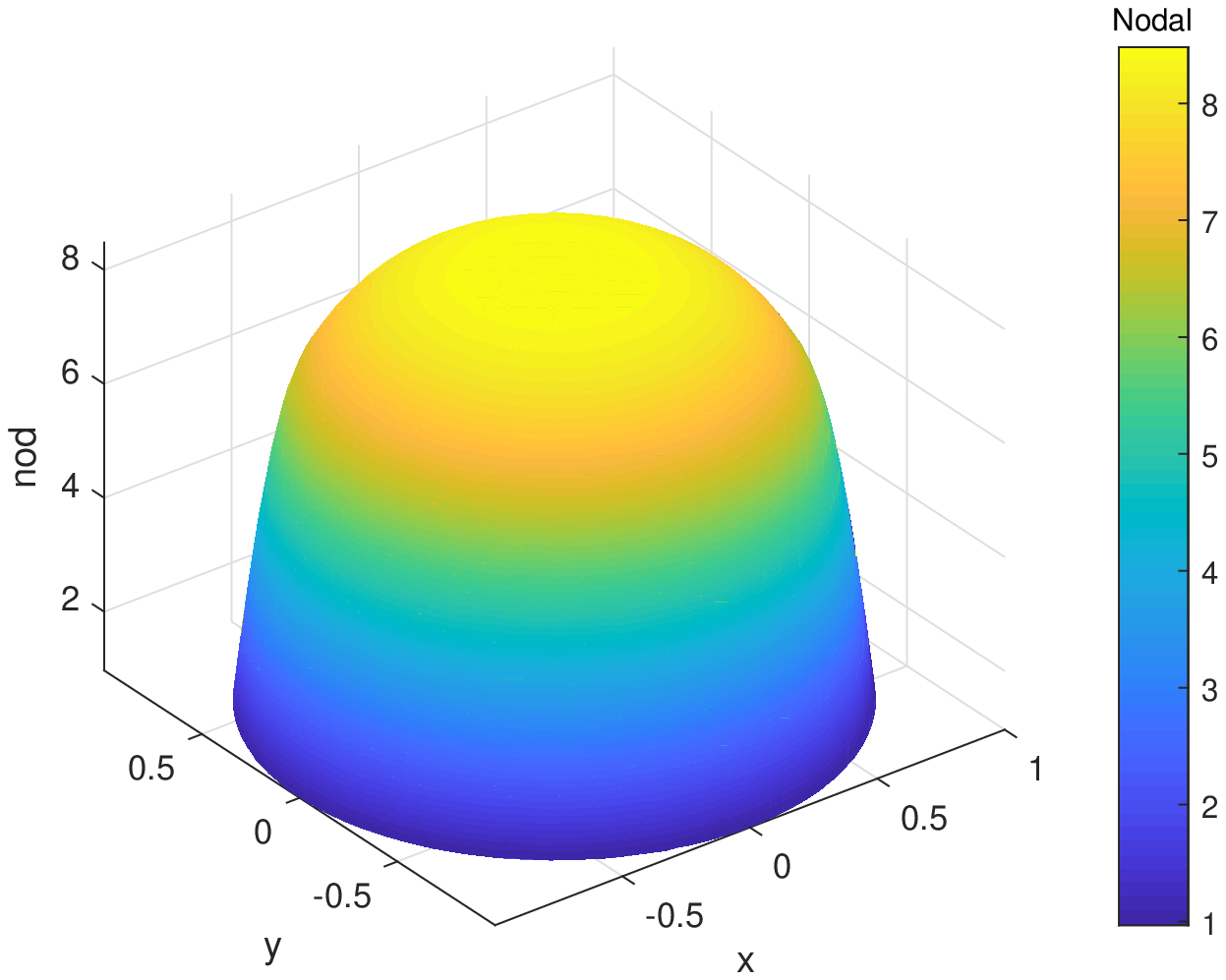}
\caption{Steady state concentrations of BMP4, Wnt and Nodal at time $t=3$ days.}
\label{fig:2d3d}
\end{figure}

\begin{figure}[t]
\centerline{\includegraphics[scale=0.5]{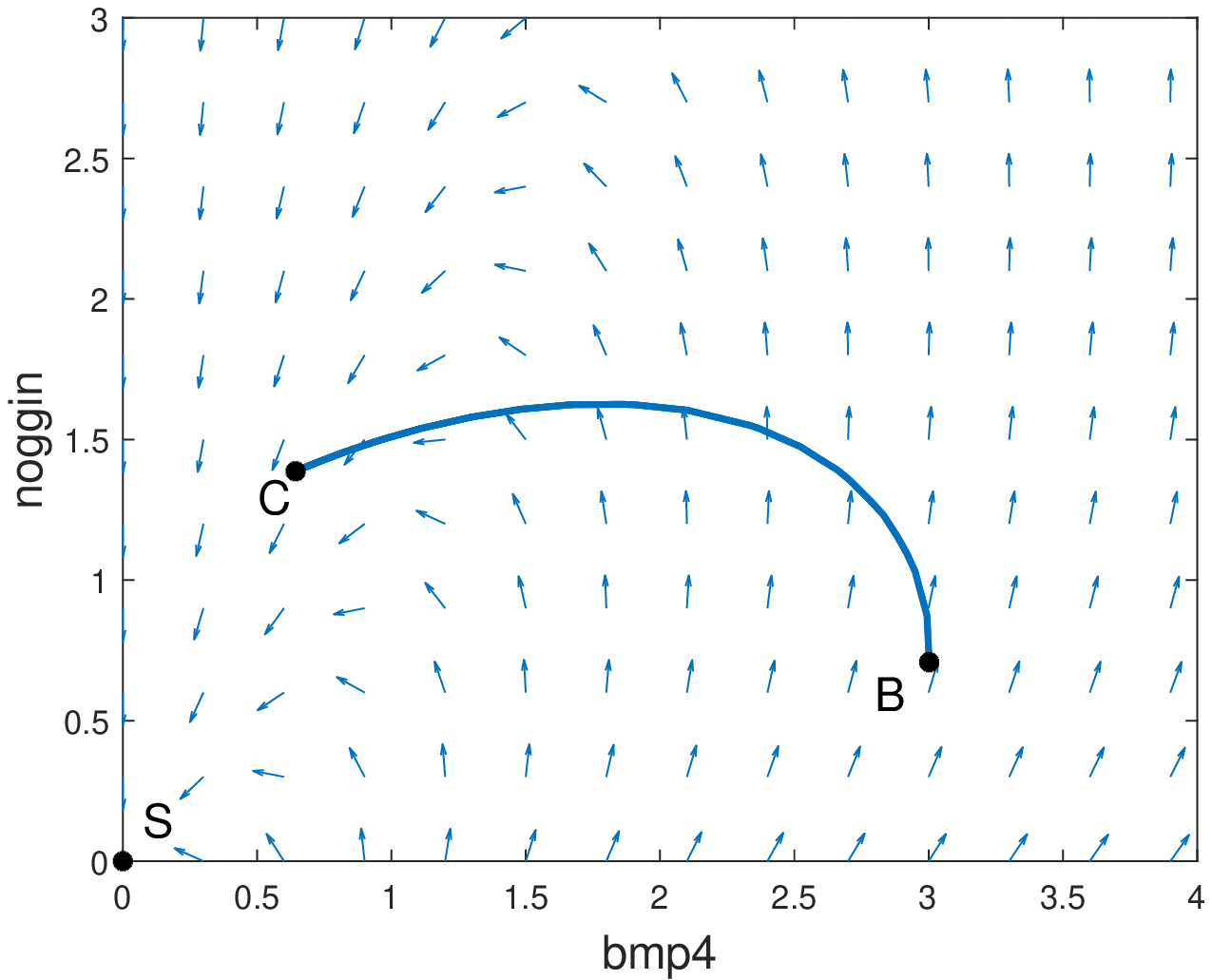}
\includegraphics[scale=0.5]{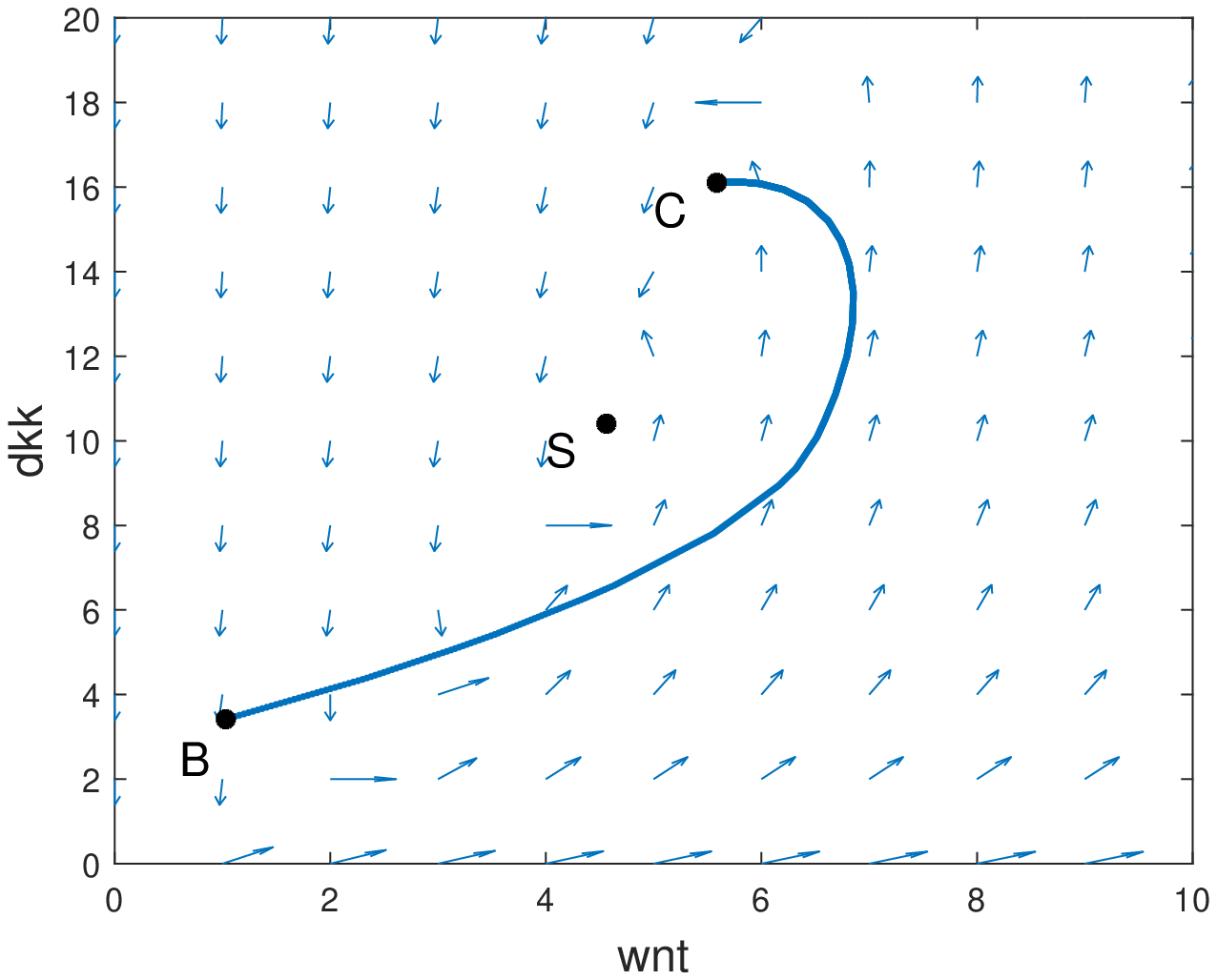}}
\centerline{\includegraphics[scale=0.5]{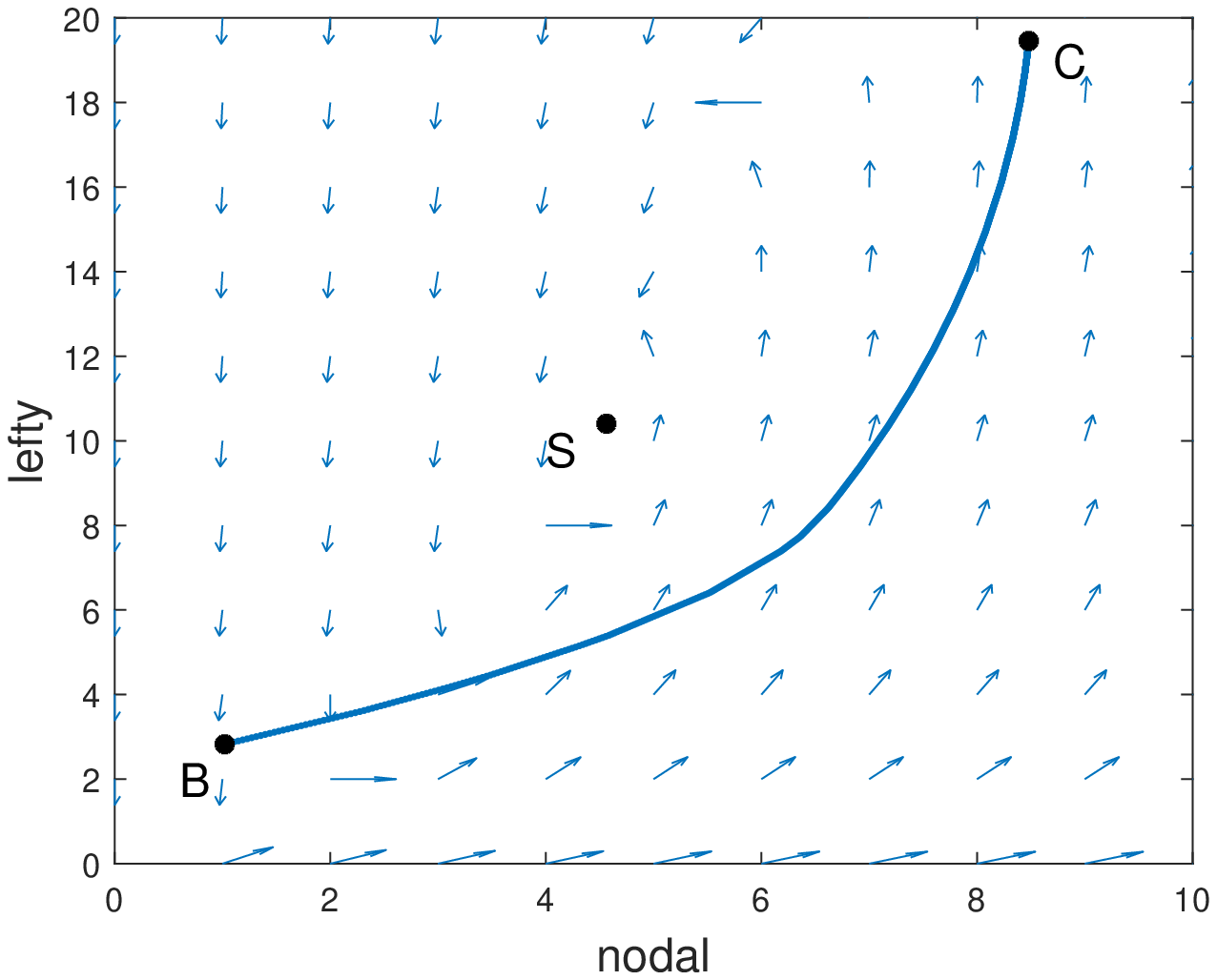}}
\caption{Radial sections of three pairs of activator-inhibitors in the phase plane. The plots show the values of bmp4/noggin, wnt/dkk, and nodal/lefty from the boundary $B$ to the center of the colony $C,$
when the concentrations reach steady states, at $t=3$ days. On the top plot, $S$ is the stable node for the reaction dynamics. On the middle and bottom plots, $S$ is a stable focus with counterclockwise rotation. Plots also show velocity fields of each reaction system.}
\label{fig:phase_plane}
\end{figure}

\begin{figure}[t]
\centerline{\includegraphics[scale=0.55]{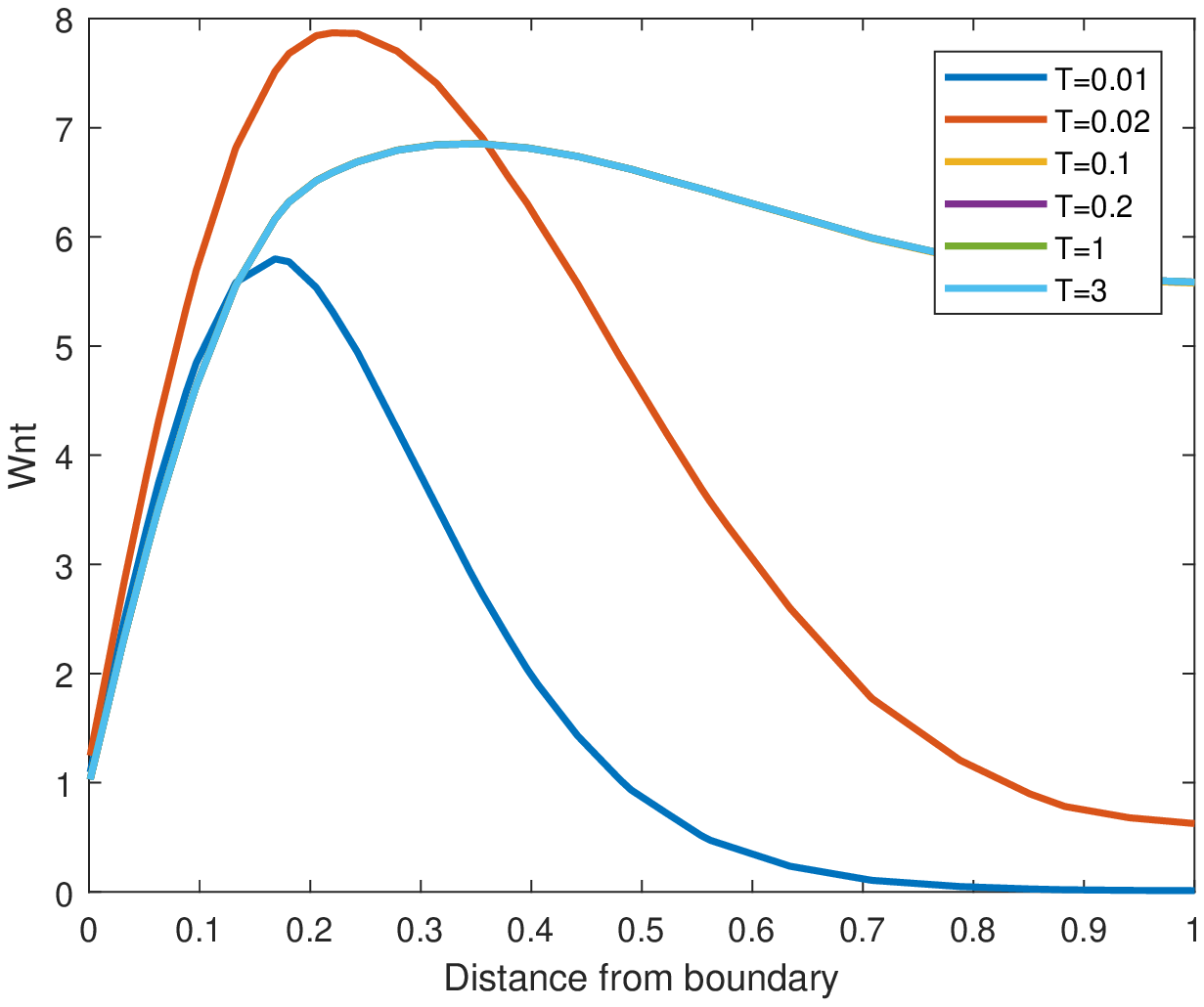}
\includegraphics[scale=0.55]{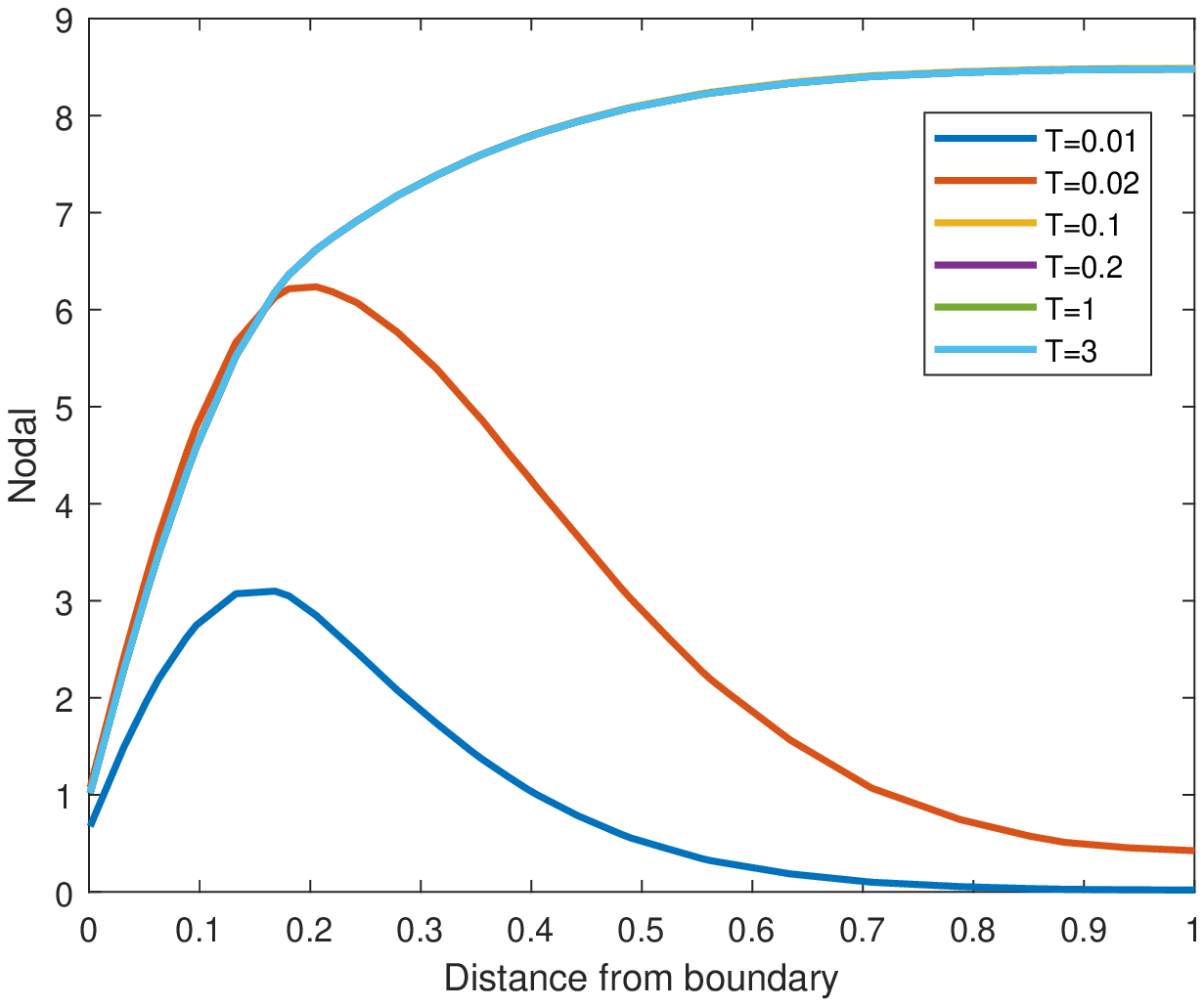}}
\caption{Signaling waves of Wnt and Nodal. Plots represent radial profile of the proteins at increasing moments of time $T.$ Initial concentrations are zeros for both proteins.}
\label{fig:T_cuts}
\end{figure}

\begin{figure}[t]
\centerline{\includegraphics[scale=0.55]{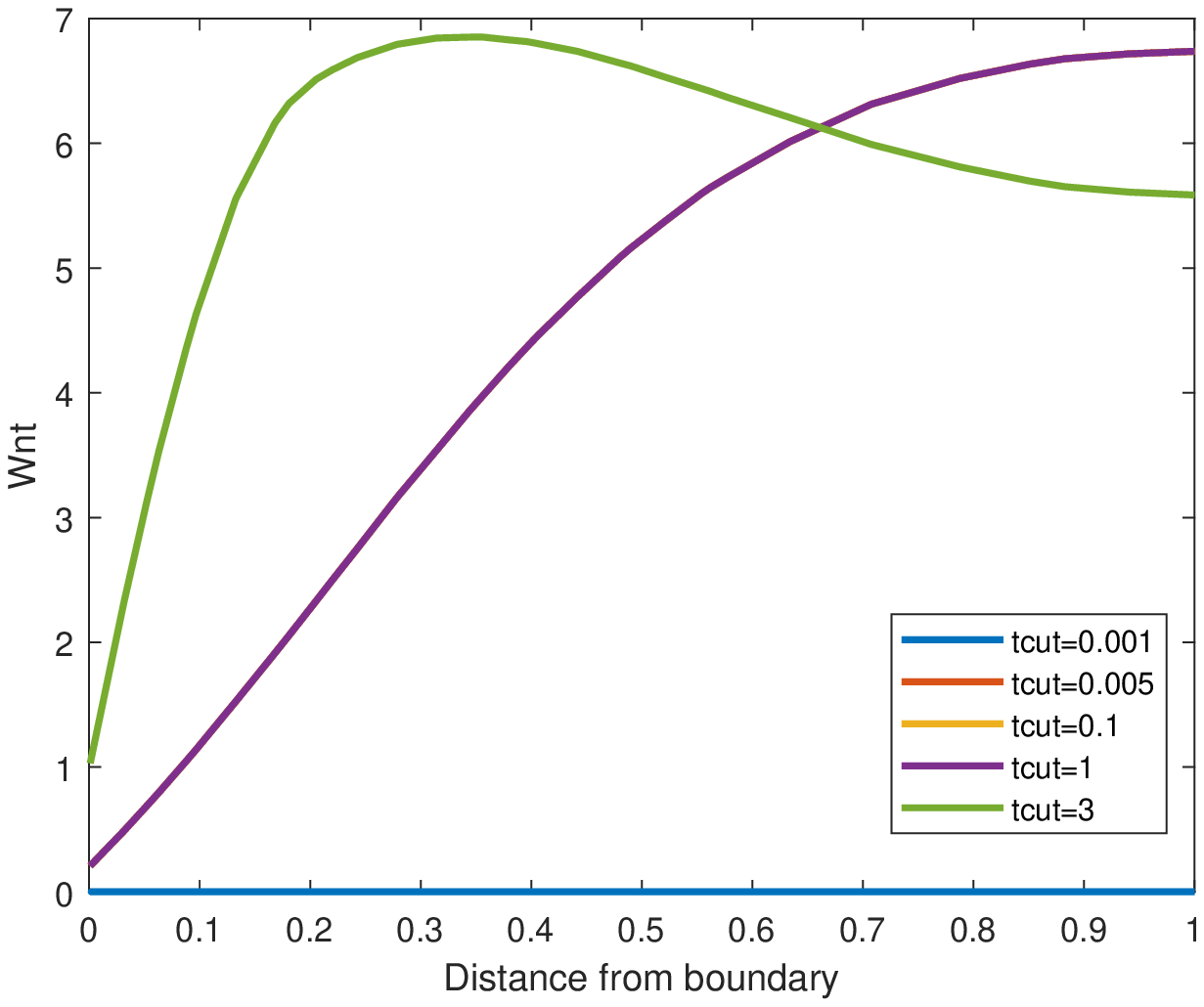}
\includegraphics[scale=0.55]{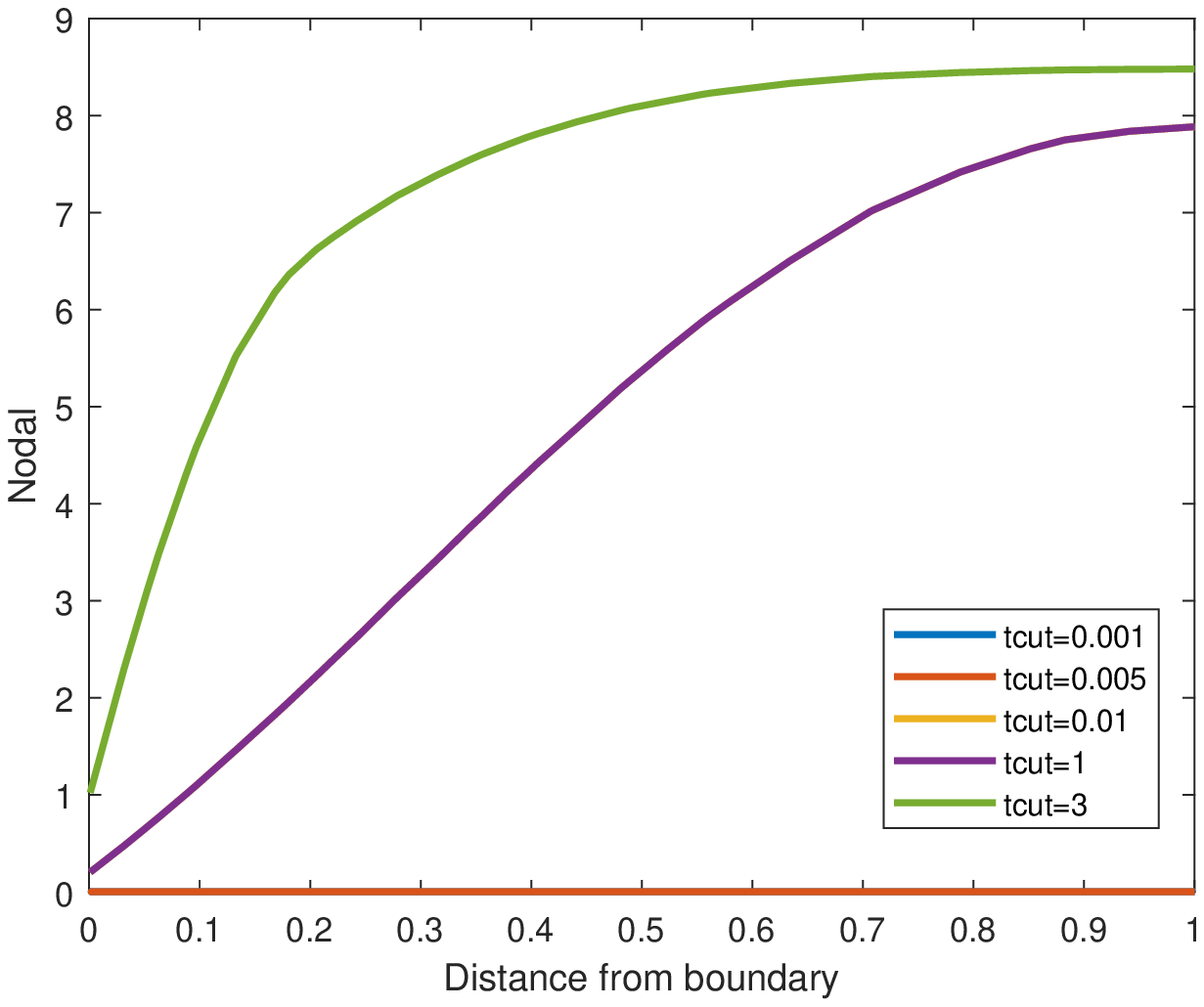}}
\caption{Effect of BMP inhibition at time $tcut$ on the terminal concentrations of Wnt and Nodal. Concentrations with $tcut=3$ correspond to no inhibition of BMP4. Plots show switching from a zero steady state to a non-zero state, when the activation time of BMP4 exceeds certain threshold value, but it is inhibited afterwards.}
\label{fig:t_cuts}
\end{figure}

\begin{figure}[H]
\centering
\includegraphics[scale=0.55]{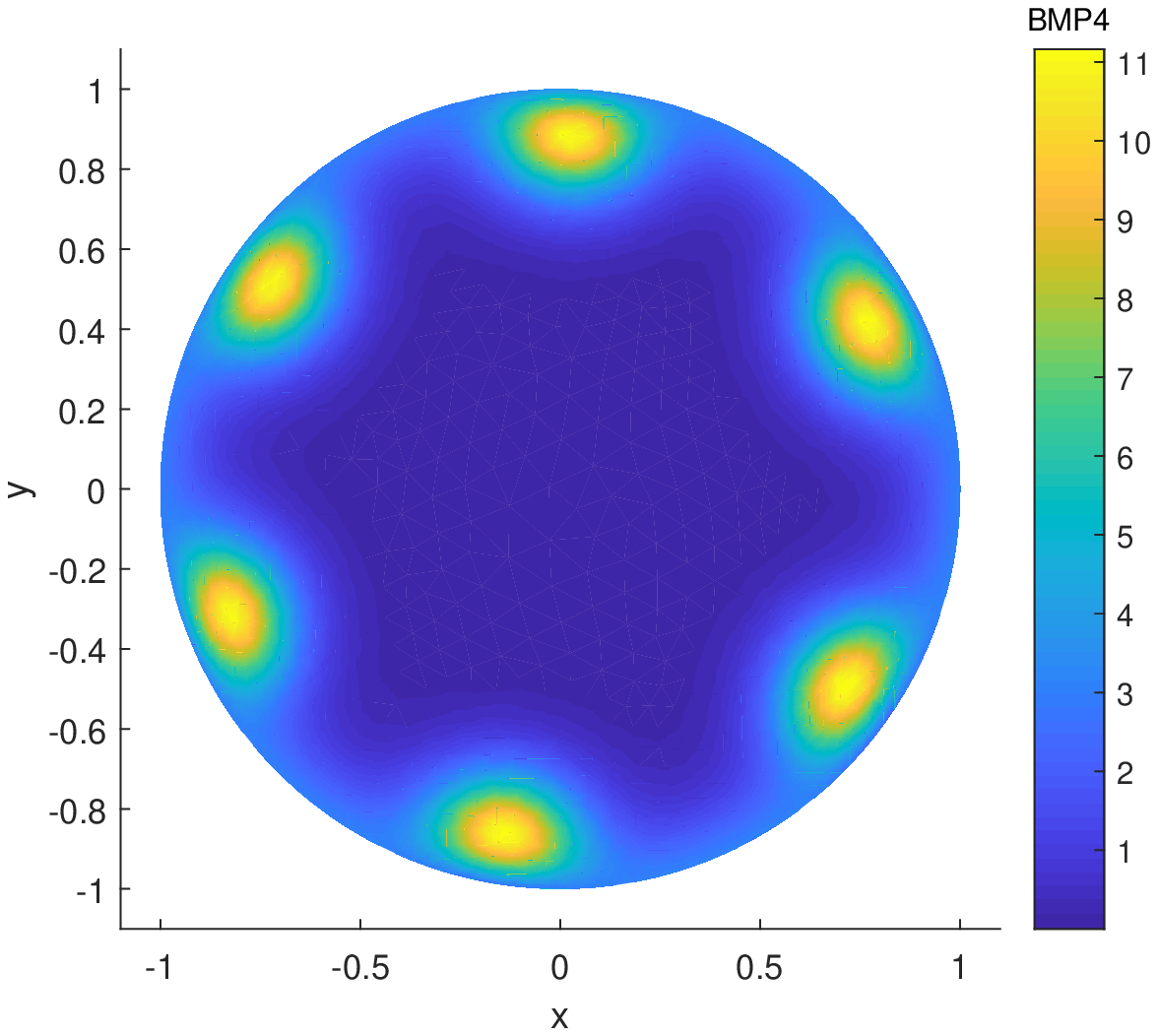}
\includegraphics[scale=0.55]{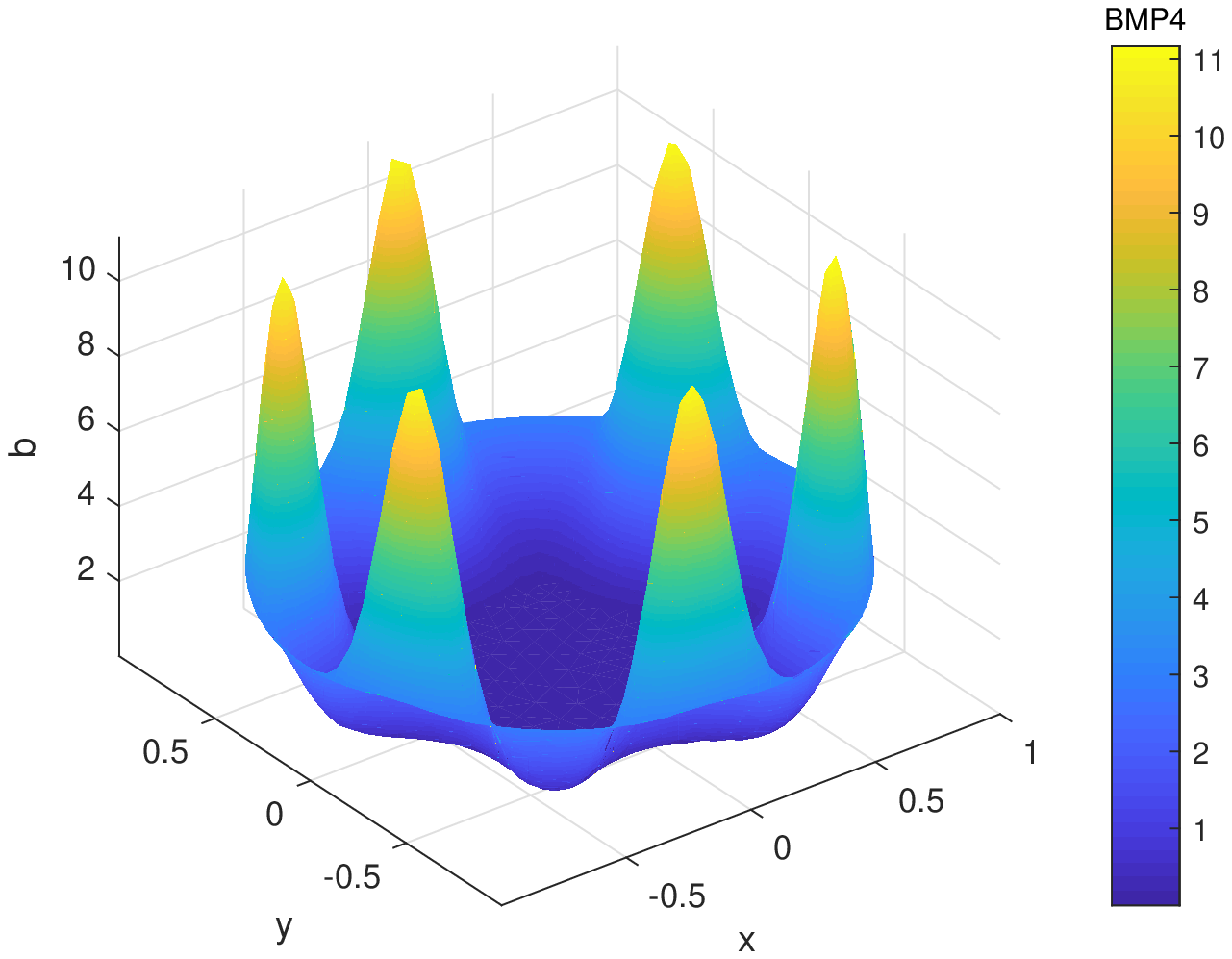}
\caption{Instabilities in BMP4/Noggin dynamics. The figure shows 2d and 3d plots of the result of the numerical simulation of concentration of BMP4 at $t=3$ for the diffusion coefficients $\mu_{bmp}=1$ and $\mu_{nog}=55$ and the colony size of $500\mu m.$}
\label{fig:t2d}
\end{figure}

\section{Existence  and stability of steady-state solutions}
One of the main features of the models considered in this paper and experimental works cited in the introduction is the formation non-homogeneous steady-states of the chemical concentrations. Moreover, this steady-states appear to be stable, as none of the Turing-type instabilities is observed in the experimental setting, see Chhabra et al \cite{Chhabra}.
In this section we address the question of existence of stable steady-state solutions. The results that we prove below, apply to RD systems in the form:
\begin{eqnarray}
\label{eq:u1}
\partial_t u{}-{}\mu\Delta u{}&=&{}-au {}+{}\frac{bu^2}{1+v}{}+{}f(x),\\
\label{eq:v1}
\partial_t v{}-{}\mu\Delta v{}&=&{}-cv{}+{}du^2{}+{}g(x),
\end{eqnarray}
with the Direchlet boundary conditions 
\begin{equation}
\label{bc:exist}
u{}={}u_b(x),\quad v{}={}v_b(x),\quad (x,t)\in\partial\Omega\times[0,+\infty),
\end{equation}
This can be seen as a limiting case of the boundary conditions from earlier sections, when the rates of cooling $h_u,\,h_v\to+\infty,$ i.e., there is high rate of transfer of chemicals to or from the background state. We will assume that decay and reaction coefficients are positive.

We would like to compare the method we use with the well-known method of invariant regions by Chueh et al. \cite{ChCS}, for establishing time asymptotic behavior of  solutions of reaction-diffusion systems.  The latter method applies to reaction-diffusion systems in the form
\begin{equation}
\label{eq:inv}
\frac{\partial u}{\partial t}{}-{}D\Delta u{}={}F(u,t),
\end{equation}
where $u\in \mathbb{R}^n,$ $(x,t)\in \Omega\times [0,+\infty),$ $D$ is $n\times n,$ diagonal matrix with non-negative entries,  and the vector source term  $F\in\mathbb{R}^n.$
The system is supplied with the zero-flux boundary conditions
\[
\frac{\partial u}{\partial n}{}={}0,\quad (x,t)\in\partial\Omega\times[0,+\infty).
\]
For the method to work the must have bounded invariant regions (see below) in order to establish  bounds on supremum norm of $|u(x,t)|.$ 
Next,  the method relies on the fact  that time asymptotic behavior of solutions of \eqref{eq:inv} can be compared with the solution of the system of ODEs:
\[
\frac{du}{dt}{}={}F(u,t),
\]
meaning that the limiting behavior of \eqref{eq:inv} is a homogeneous (constant in $x$) state. 

If we look at the system \eqref{eq:u1}, \eqref{eq:v1}, we see that the right-hand side  explicitly depends on $x$ through functions $f(x)$ and $g(x).$ The boundary conditions \eqref{bc:exist} differ as well. Thus, in general,  the steady-states  of the problem \eqref{eq:u1}, \eqref{eq:v1} and \eqref{bc:exist} are  non homogeneous (non-constant).

We show now that system \eqref{eq:u1}, \eqref{eq:v1} does not have invariant regions either, so that the uniform estimates must be obtained by other means.
 
A closed set $S\subset \mathbb{R}^n$ is called an invariant set for \eqref{eq:inv} if for any $t>0,\,x\in\Omega$ the solution $u(x,t)\in S$ whenever the initial data $u_0(x)\in S,$ for every $x\in\Omega.$

Suppose that $S$ can be written as an intersection of ``half-spaces'':
\[
S{}={}\cap_{i=1}^m \left\{ G_i(u){}\leq {}0\right\}
\]
where $G_i$ are smooth functions.
Theorem 14.14 from \cite{Smoller} gives a sufficient and necessary conditions for $S$ be an invariant regions. 
\begin{theorem*}
$S$ is an invariant region for \eqref{eq:inv}, if and only if  for every $u\in\partial S,$ (so that $G_i(u){}={}0$, for some $i,$):
\begin{enumerate}
\item $\nabla G_i(u)$ is the left eigenvector of $D;$
\item $G_i$ is quasi-convex at $u;$
\item $\nabla G_i(u)\cdot F(u,t)\leq 0.$
\end{enumerate}
\end{theorem*} 
Recall that function $G_i(u)$ is called a quasi-convex function at $u,$ if for any vector $v$ such that $\nabla G_i(u)\cdot v{}={}0,$ we have $\langle \grad^2 G_i(u)v,v\rangle\geq 0.$

Consider now system \eqref{eq:u1}, \eqref{eq:v1}. In this case $D$ is a diagonal matrix $\mu\mathbb{I},$ where $\mathbb{I}$ is $2\times2$ identity matrix. The first condition of the theorem implies that $\nabla G_i(u)$ is proportional to vector $(1,0)$ or $(0,1),$ that is, the level sets of $G_i$ are either horizontal or vertical lines in $(u,v)$ plane. Thus, an invariant region, if it exists, is a rectangle.  The third condition of the theorem then implies that on a line $\{ u = const.\},$ we must have
\[
-cv{}+{}du^2{}\leq{}-g(x){}\leq0,
\]
and on a line $\{v= const.\},$
\[
-au{}+{}\frac{bu^2}{1+v}{}\leq{}-f(x){}\leq 0.
\] 
It can be seen from Figure \ref{fig:nul} that it is impossible as non of these functions changes its sign.

\subsection{Main theorem}
We will use the standard notation for the spaces of continuous,  H\''{o}lder continuous functions, as well as $L^p$ spaces. For definitions, we refer readers to \cite{Ladyzh}. Norms in $L^p(\Omega)$ space  will denoted by $\|u\|_p,$ $1\leq p\leq \infty.$ We let $Q_T=\Omega\times(0,T)$ and $\Gamma_T= (\Omega\times\{t=0\}) \cup (\partial \Omega  \times [0,T]),$ for $T>0.$

Now we state our main result, that we will be proved below.
\begin{theorem}
Let $\Omega$ be an open, bounded,  set with $C^{2+\alpha}$ boundary and $\alpha\in(0,1).$   Let $f(x),g(x)\in C^\alpha(\overline{\Omega})$ be non-negative  functions, $u_0, v_0 \in C^{2+\alpha} (\overline{\Omega}),$ $u_b,v_b\in C^{2+\alpha}(\partial\Omega),$ and  necessary compatibility conditions between the initial and boundary values hold.
Then, there exists a unique classical solution of the system on $\Omega \times [0,\infty).$ For any $T>0,$
$u,v\in C^{2+\alpha,1+\alpha/2}(Q_T),$  
the following properties hold.
\begin{enumerate}
    \item $u(x,t), v(x,t)$ are non-negative and bounded above with a constant independent of $x, t,$ and $\mu.$
    \item 
    There is $C_s$-- a polynomial function of $\|u_0\|_\infty,$ $\|v_0\|_\infty,$ $\|f\|_\infty$ and $\|g\|_\infty,$ independent of $\mu,$
    such that if an inequality
    \begin{equation}
    \label{th:stab}        
    C_s< \mu C(\Omega),
    \end{equation}
    holds, where $C(\Omega)$ is a constant from the Poincare's inequality, then, there exists a steady-solution $(u_s(x),v_s(x))$ of \eqref{eq:u1}, \eqref{eq:v1}, and 
    \[
    \lim_{t\to\infty}\|u(x,t)-u_s(x)\|_{L^2}{}+{}
    \|v(x,t)-v_s(x)\|_{L^2} {}={}0.
    \]
    \item Under condition \eqref{th:stab}, there is a ball $B=B(f,g,\mu)\subset L^\infty(\Omega)$ such that if the initial data $\tilde{u}_0,\tilde{v}_0\in B$ and verify all other properties of the initial data stated above, then for the corresponding classical solution $\tilde{u}(x,t),\tilde{v}(x,t),$
    \[
    \lim_{t\to\infty}\|\tilde{u}(x,t)-u_s(x)\|_{L^2}{}+{}
    \|\tilde{v}(x,t)-v_s(x)\|_{L^2} {}={}0.
    \]
 
\end{enumerate}
\end{theorem}

\subsection{Proof of main theorem}
The proof is given below in a series of lemmas, where  $(u,v)$ is a local, classical solution of the problem.  We will make a repeated use the  of the maximum  principle, that can be found,  for example, the book by Evans \cite{evans}. 
\begin{lemma}
Let $w\in C^{1,2}(Q_T)\cap C(\overline{Q_T})$ be a function that satisfies,
\begin{equation*} \partial_t{w} - \mu \Delta{w} \leq(\geq) -kw,\end{equation*}
where $\mu, k\geq 0.$ Then, 
 \begin{equation*}
 \max\limits_{Q_T} w(x,t) = \max\limits_{\Gamma_T} w(x,t), \label{eq:genmaxprinciple}\end{equation*}
or, if the inequality is reversed,
 \begin{equation*} 
 \min\limits_{Q_T} w(x,t) = \min\limits_{\Gamma_T} w(x,t). \label{eq:genminprinciple}
\end{equation*}    
\end{lemma}

Since $f(x),g(x)\geq 0,$ it follow from this lemma that $u(x,t)$ and $v(x,t)$ are non-negative functions.
We proceed with uniform upper bounds.
\begin{lemma}
\label{lemma:uniform}
There is $C>0,$ depending on $a,b,c,d,$ $\max u_0,$ $\max v_0,$ $\max f(x)$ and $\max g(x),$ but not $\mu,$
such that
\begin{equation}
\label{thmubd}
\max\limits_{(x,t)\in Q_T} u(x,t),\,\max\limits_{(x,t)\in Q_T} v(x,t)  \\ \leq C \left(\|u_0\|_\infty, \|v_0\|_\infty,\|f\|_{\infty}, \|g\|_\infty\right),
\end{equation} 
where  $C$ is polynomial function of its arguments, with positive coefficients, and is independent of $\mu.$
\end{lemma}
\begin{proof}
Let $Z$ be a solution of the ordinary differential equation, 
\begin{equation*}
    \partial_t{Z} = -aZ - \|f\|_{\infty},
\end{equation*}
with initial condition $Z(0) = 0,$ i.e., $Z(t) = \left(e^{-at} - 1\right) \frac{\|f\|_\infty}{a}$. Setting $\hat{u}(x,t) = u(x,t) + Z(t)$ we find that
\begin{equation}
\label{eq:uZ}
    \partial_t{\hat{u}} - \mu \Delta{\hat{u}}  = -a \hat{u} + \frac{b (\hat{u}-Z)^2}{1+v} + f(x) - \|f\|_\infty\\
    \leq -a \hat{u} + \frac{b (\hat{u}-Z)^2}{1+v}.
\end{equation}
As for $v$, we have,
\begin{equation}
\label{eq:vZ}
    \partial_t{v} - \mu \Delta{v}  = -cv + d (\hat{u}-Z)^2 +g(x)
    \geq -cv + d (\hat{u}-Z)^2.
\end{equation}
Let $\phi$ be a smooth, non-increasing function that will be chosen later.
For function $\phi(v(x,t))$ we obtain 
\[
    \partial_t{\phi(v)} - \mu \Delta\phi(v)  + \mu \phi''(v) |\grad v|^2{}+{}cv\phi'(v) - d(\uh-Z)^2 \phi'(v) \leq 0.
\]
Adding the last equation to  \eqref{eq:uZ} we get,
\begin{multline*}
    \partial_t{\left(\uh + \phi(v)\right)} - \mu \left( \Delta{\uh} + \Delta{\phi(v)} \right) + \phi''(v) |\grad v|^2  + a\uh + cv\phi'(v)
    -(\uh-Z)^2\left(\frac{b}{1+v} + d \phi'(v) \right) \leq 0.
\end{multline*}
We will select $\phi{}={}-\frac{b}{d}\ln(1+v),$ so that
\begin{multline}
\label{eq:uvphi}
\partial_t{\left(\uh + \phi(v)\right)} - \mu \Delta{\left(\uh + \phi(v)\right)} \leq - a (\uh+\phi(v)) + a \phi(v) + \frac{bc}{d} \frac{v}{1+v}\\
\leq -a(\uh +\phi(v)) + \alpha\max_{Q_T}\ln(1+v),
\end{multline}
for some $\alpha$ depending on $a,b,c,d.$
Let $W$ be a solution of
\begin{equation*}\partial_t{W} = -aW + \alpha\max_{Q_T}\ln(1+v), \end{equation*}
with initial condition $W(0)=0,$ i.e., 
$W(t) = \alpha_1\max_{Q_T}\ln(1+v)\left( 1 - e^{-at} \right),$ with $\alpha_1=\alpha/a.$  Subtracting equation for $W$ from  \eqref{eq:uvphi}, we obtain:
\begin{equation*}
\partial_t{\left( \uh + \phi(v) -W\right)} - \mu \Delta{\left( \uh + \phi(v) -W\right)}  
 \leq -a \left( \uh + \phi(v) -W\right).
\end{equation*}
Now, using the maximum principle \eqref{eq:genmaxprinciple} we obtain
\begin{multline*}
\max_{Q_T} (\uh+\phi-W)  = \max\left( \max_{\Omega\times \{0\}} (\uh+\phi), \max_{\Gamma_T} (Z(t)+\phi(v)-W(t))\right)\\
 \leq \max\left(\max_{\Gamma_T} u_b,\, \max_{\Omega} (u_0+Z(0)+\phi_0)\right)
 \leq \max_{\Omega}u_0.
\end{multline*}
Therefore, for any $(x,t)$ in the domain $Q_T,$
$\uh+\phi-W \leq \max_{\Omega}u_0,$ or,
\begin{equation}
\label{eq:uh}
\uh(x,t)  \leq \frac{b}{d}\ln(1+v(x,t))+W(t) \leq \max_\Omega u_0{}+{}\alpha_2\ln(1+\max_{Q_T}v),
\end{equation}
for some $\alpha_2$ depending on $a,b,c,d.$

 Consider now equation \eqref{eq:vZ}. 
 \begin{eqnarray*}
 \partial_t{v} - \mu \Delta{v}  &=& -cv + d (\uh-Z)^2 + g(x)     \\
  & \leq &-cv + d \left(\max_\Omega u_0{}+{}\alpha_2\ln(1+\max_{Q_T}v) +\frac{\|f\|_\infty}{a} \right)^2 + \|g\|_{\infty}.
\end{eqnarray*}
Using a maximum principle again we get
\[
\max_{Q_T} v \leq \max_\Omega v_0 {}+{}\frac{d}{a} \left(\max_\Omega u_0{}+{}\alpha_2\ln(1+\max_{Q_T}v) +\frac{\|f\|_\infty}{a} \right)^2 + \frac{1}{a}\|g\|_{\infty}.
\]

By the elementary properties of function $\ln(1+v),$ we find that $\max_{Q_T} v$
is bounded by a polynomial with positive coefficients in variables of $\max u_0,$ $\max v_0,$ $\| f\|_\infty,$ $\|g\|_\infty.$ The corresponding estimate for $\max_{Q_T} u$ follows from this and \eqref{eq:uh}.
\end{proof}
Now, the global existence follows.
\begin{lemma}The unique, local, classical solution $(u,v)$ can be extended for all times $t>0.$\
\end{lemma}
\begin{proof}
We will use the following characterization of time maximal time of existence $T$ of a local solution, from Rothe \cite{Rothe}, theorem 1, page 111. It is proved there that if $T<+\infty$ then the $\max$--norm over $x$ of $(u(x,t),v(x,t))$ grows without bound as $t$ approaches $T$. But this can not happen due to the estimates derived above in \eqref{thmubd}.  Therefore, the contradiction leads us to conclude that the classical solution in fact exists for all times $t>0.$
\end{proof}


To show that the classical solution $(u,v)$ of the reaction-diffusion system settles on a steady state it sufficient to show that the time derivative of the solution converges to zero. We will use an energy-type estimate to establish this fact. The proof makes use of the Poincare's inequality that we state for a reference below.
\begin{lemma}
\label{lemma:ut}
Let $\Omega$ be a bounded, connected, open subset of $\mathbb{R}^n$ with a $C^1$ boundary $\partial \Omega$. Let $1\leq p < \infty$. Then there exists a constant $C$, depending only on $n,p$ and $U$, such that for any integrable function $u$
with $\grad u\in L^p(\Omega)$ and zero trace on the boundary $\partial \Omega,$ 
\begin{equation}
\|u\|_{L^p(\Omega)} \leq C(\Omega,p)\|\grad u\|_{L^p(\Omega)}.   \label{thm:Poincare}
\end{equation}
\end{lemma}
Proof can be found in  Brezis book \cite{Brezis}.

\begin{lemma}
\label{lemma:stab}
For all $t\in[0,T]$, $T>0,$ it holds:  
\begin{equation*}
    \frac{d}{dt} \left( \|u_t\|_2^2 + \|v_t\|_2^2 \right) + 4(\mu C(\Omega) - C_s) \left( \|u_t\|_2^2 + \|v_t\|_2^2\right) \leq 0,
    \end{equation*}
    where  $C_s=C_s(\max_{Q_T}u,\max_{Q_T}v)$ is a polynomial function of its arguments, and is  independent of $\mu,$ and $T,$  and      $C(\Omega)$ is a constant from the Poincare's inequality.
   If the stability condition 
   \begin{equation}
   \label{eq:stab}
   C_s < \mu C(\Omega)
   \end{equation}
   holds, then   
   \begin{equation*} \|u_t(x,t)\|_2^2 + \|v_t(x,t)\|_2^2 \leq e^{-4(\mu C(\Omega) - C_s) t} \left(\|u_0\|^2_{C^2(\Omega)} + \|v_0\|^2_{C^2(\Omega)}\right)\to0, 
 \end{equation*}
 as $t\to+\infty.$
\end{lemma}
\begin{proof}
Taking time derivative of both sides of the equation for $u$ \eqref{eq:u}, we get,
\begin{equation*}\partial_t{u_t} - \mu \Delta{u_t} = -a u_t + 2b\ \frac{u\ u_t}{(1+v)} -b\ \frac{u^2\ v_t}{(1+v)^2}.\end{equation*}
Multiply with $u_t$ and integrate over the domain to get,
\begin{equation*}
    \int_\Omega\ u_t \partial_t{u_t}\ dx - \mu \int_\Omega\ u_t \Delta{u_t} \ dx = -a  \int_\Omega\ |u_t|^2 \ dx + 2b\  \int_\Omega\ \frac{u\ |u_t|^2}{(1+v)}\ dx -b\  \int_\Omega\ \frac{u^2\ u_t\ v_t}{(1+v)^2}\ dx.
\end{equation*}
Using integration by parts,
\begin{multline*}
    \frac{1}{2} \frac{d}{dt}\int_\Omega\ |u_t|^2 \ dx - \mu \int_\Gamma\ u_t \nabla{u_t}\cdot n \ dl + \mu \int_\Omega\ |\nabla{u_t}|^2 \ dx \\= -a  \int_\Omega\ |u_t|^2 \ dx + 2b\  \int_\Omega\ \frac{u\ |u_t|^2}{(1+v)}\ dx -b\  \int_\Omega\ \frac{u^2\ u_t\ v_t}{(1+v)^2}\ dx.
\end{multline*}
We will use Young's inequatlity, uniform bounds on $u$ and $v,$ and   Poincar\'e's inequality \eqref{thm:Poincare} with $p=2$ applied to the function $\partial_t u$ (notice, that $\partial_t u$ equals to zero on the boundary of the domain) to get the next result. The constant $C(\Omega,2)$ from that inequality will be abbreviated to $C(\omega).$
\begin{equation} 
\frac{1}{2} \frac{d}{dt}\int_\Omega\ |u_t|^2 \ dx \leq  \left(-\mu\ C(\Omega) + c_0\right)  \int_\Omega\ |u_t|^2{}+{} |v_t|^2 \ dx, 
\label{eq:dtutsq}
\end{equation}
where $c_0$ has polynomial dependence on $\max u,$ $\max v.$
Similarly, for $v$,
\begin{equation*}
\frac{1}{2} \frac{d}{dt}\int_\Omega\ |v_t|^2 \ dx \leq 
\left(-\mu\ C(\Omega) + c_0\right)  \int_\Omega\ |u_t|^2{}+{} |v_t|^2 \ dx.
\end{equation*}
\begin{eqnarray}
    \frac{1}{2} \frac{d}{dt}\int_\Omega\  |u_t|^2 + |v_t|^2 \ dx & \leq 2\left(-\mu\ C(\Omega) + c_0\right) \int_\Omega\  |u_t|^2+|v_t|^2 \ dx.
\end{eqnarray}
Now the statement of the lemma follows from the last inequality and uniform bounds from lemma \ref{lemma:uniform}
\end{proof}


When uniform estimates from lemma \ref{lemma:uniform} are substituted into function $C_s$ in lemma \ref{lemma:stab}, condition \eqref{eq:stab} defines the range of $L^\infty$ norms of admissible data $\|f\|_\infty,$ $\|g\|_\infty,$ $\|u_0\|_\infty,$ $\|v_0\|_\infty.$ We can state this in another way, by saying that if $f$ and $g$ are such 
that stability condition \eqref{th:stab} holds with $u_0=0$ and $v_0=0,$ then there is a ball 
$B=B(f,g,\mu)\subset L^\infty(\Omega),$ centered at zero such that the same  condition \eqref{th:stab} holds for any $u_0,v_0\in B.$ In the remaining part of the proof we will assume this condition.

Next we will obtain bound on the gradients of $\grad u, \grad v.$
\begin{lemma} For  any $t>0,$ 
\begin{equation*}\mu \int_\Omega\left( |\nabla u(x,t)|^2 + |\nabla v(x,t)|^2\right) \ dx \leq C(\|f\|_{\infty},\|g\|_{\infty},\|u_0\|_{C^2},\|v_0\|_{C^2}), 
\end{equation*}
with a positive $C,$ independent of time.
\end{lemma}
\begin{proof}
We multiply for $u$ by $u_t$ and integrate by parts to get,
\begin{equation*}
\int_\Omega\ |u_t|^2\ dx - \mu \int_\Gamma u_t \nabla u\cdot n\ d\sigma + \mu \int_\Omega\ \nabla u \cdot \nabla u_t\ dx= -a\int_\Omega\ u u_t\ dx + b \int_\Omega\ \frac{u^2}{1+v} u_t\ dx + \int_\Omega\ f(x) u_t\ dx. \end{equation*}
Using boundary condition on $u_t$ and uniform bounds on $u$ and $v$ from lemma \ref{lemma:uniform},
\begin{eqnarray*}
\int_\Omega\ |u_t|^2\ dx + \mu \frac{1}{2}\frac{d}{dt} \int_\Omega\ |\nabla u|^2 \ dx & = -a\int_\Omega\ u u_t\ dx + b \int_\Omega\ \frac{u^2}{1+v} u_t\ dx + \int_\Omega\ f(x) u_t\ dx\\
& \leq \left( c_0+ \|f\|_\infty \right) \int_\Omega\ |u_t|\ dx = c_1 \|u_t\|_2,
\end{eqnarray*}
with an appropriate  $c_0.$
Similarly, for $v$, we get,
\begin{equation*}\int_\Omega |v_t|^2\ dx + \mu \frac{1}{2} \frac{d}{dt} \int_\Omega|\nabla v|^2 \ dx  \leq c_1 \|v_t\|_2.
\end{equation*} 
Adding the last two inequalities we get,
\begin{equation*}
    \mu \frac{d}{dt} \int_\Omega\ \left( |\nabla u|^2 + |\nabla v|^2\right) \ dx \leq c_1\left( \|u_t\|_2 + \|v_t\|_2\right).
\end{equation*}

We integrate this inequality in time from $0$ to $t,$ use exponential decay estimate on $\|u_t\|_2,$ $\|v_t\|_2$ from previous lemma, together with uniform bounds on $u$ and $v$ to get:
\begin{eqnarray*}
    \mu \int_\Omega\left( |\nabla u(x,t)|^2 + |\nabla v(x,t)|^2\right) \ dx & \leq \mu \int_\Omega\left( |\nabla u_0|^2 + |\nabla v_0|^2\right) \ dx + c_1 \int_0^t \left( \|u_t\|_2 + \|v_t\|_2\right)\ dt
 {}\leq{}   C,
\end{eqnarray*}
with some $C>0,$ as stated in the lemma. 
\end{proof}


Let $u(x,t),$ $v(x,t)$ be the solution of the section from previous section. Let $t_n$ be a non-decreasing sequence of time converging to $+\infty. $ Consider sequences of functions $\{u_n(x){}={}u(x,t_n)\}$ and
$\{v_n(x)=v(x,t_n)\}.$ From the estimates of $u,v$ and their gradients, it holds that there is $C$ independent of $n$ such that
\[
\|u_n\|_2{}\leq{}C,\quad \|v_n\|_2\leq C,
\]
\[
\|\grad u_n\|_2{}\leq{}C,\quad \|\grad v_n\|_2\leq C.
\]
We will need the following compactness result, the proof of which can be found in  chapter 5 of Evans book \cite{evans}.
\begin{lemma}
Assume $\Omega$ is a bounded open subset of $\mathbb{R}^n$, and $\partial \Omega$ is $C^1$. Suppose $1\leq p <n$. Then, there is a compact embedding
\begin{equation}
    W^{1,p}(\Omega) \subset \subset L^q(\Omega), \label{thm:relkon}
\end{equation}
for each $1\leq q <\frac{np}{n-p}.$ 
\end{lemma}
Since $\Omega$ is a bounded set it follows this  theorem that both sequences are pre-compact in $L^2(\Omega).$ This means that there is a subsequence of $\{t_n\},$ that we still label by $n,$
and two functions $u,v\in W^{1,2}(\Omega)$ such that 
\[
\lim_{n\to\infty}u_n{}={}u_s,\quad \lim_{n\to\infty}v_n{}={}v_s
\]
in $L^2$ norm. From the convergence in norm, it follows that a further subsequence can be extracted such that $u_n$ and $v_n$ converge to $u_s$ and $v_s$ almost everywhere in $\Omega.$\\
Moreover, since $L^2(\Omega)$ is a reflexive space and $\grad u_n,$ $\grad v_n$ are from bounded sets,
there is still further subsequence such that 
\[
\grad u_n \to \grad u_s,\quad \grad v_n \to \grad v_s,
\]
weakly in $L^2(\Omega).$

Notice, also, that as time derivatives of $u,v$ are bounded,
\[
u_t(x,t_n) \to 0,\quad v_t(x,t_n)\to 0,
\]
in $L^2$ norm.

Now we pass to the limit in the equations.
\begin{lemma}
The limiting pair of functions $(u_s,v_s)$ is a classical  solution of the system of equations:
\begin{eqnarray*}
\label{eq:s_1}
-\mu\Delta u_s{}&=&{}-a u_s + \frac{bu_s^2}{1+v_s}{}+{}f,\\
\label{eq:s_2}
-\mu \Delta v_s{}&={}&-c v_s{}+dv_s^2{}+{}g.
\end{eqnarray*}
\end{lemma}
\begin{proof}
Let $\omega(x)$ be a smooth test function, equal to zero on the boundary $\partial \Omega.$ From the original reaction-diffusion system, considered at times $t=t_n$ we obtain the following integral relations:
\begin{eqnarray*}
\int u_t(x,t_n)\omega(x)\,dx {}+{}\mu\int \grad u_n\cdot\grad \omega(x)\,dx &=&
\int \left( -a u_n + \frac{bu_n^2}{1+v_n}{}+{}f(x)\right)\omega(x)\,dx,\\
\int v_t(x,t_n)\omega(x)\,dx {}+{}\mu \int \grad v_n\cdot\grad \omega(x)\,dx &=&
\int \left( -c v_n{}+dv_n^2{}+{}g(x)\right)\omega(x)\,dx.
\end{eqnarray*}
Passing to the limit in each term of these equations, using above compactness properties we obtain that
\begin{eqnarray*}
\mu\int \grad u_s\cdot\grad \omega\,dx &=&
\int \left( -a u_s + \frac{bu_s^2}{1+v_s}{}+{}f\right)\omega\,dx,\\
\mu \int \grad v_s\cdot\grad \omega\,dx &=&
\int \left( -c v_s{}+dv_s^2{}+{}g\right)\omega\,dx.
\end{eqnarray*}
i.e, $(u_s,v_s)$ is a weak solution. As a pointwise limit of $u(x,t_n),$ $v(x,t_n),$ $(u_s,v_s)$ take boundary values $u_b$ and $v_b.$ By the well know regularity results for elliptic equations, it follows that $u_s,v_s\in C^{2+\alpha}(\overline{\Omega}),$ and it is classical solutions of the same system.
\end{proof}
Now we prove the following. 
\begin{lemma}
As $t\rightarrow \infty$,  $u(x,t), v(x,t)$ of the system converges to $u_s(x), v_s(x)$ in $L_2$ norm:
\[
\lim_{t\to\infty}\| u(x,t)-u_s(x)\|_2{}={}0,\quad 
\lim_{t\to\infty}\| v(x,t)-v_s(x)\|_2{}={}0.
\]
\end{lemma}
\begin{proof}
Suppose that $(u(x,t),v(x,t))$ does not converge to $(u_s,v_s)$ in $L^2$ norm. Then, there is a sequence of times $t_n$ and $\epsilon>0$ such that
\[
\|(u(x,t_n),v(x,t_n))-(u_s(x),v_s(x))\|_2\geq\epsilon.
\]
Using the arguments of this section we conclude that there is another steady-state $(\tilde{u}_s,\tilde{v}_s)$ and 
\begin{equation}
\label{eq:conv}
\|(\tilde{u}_s,\tilde{v}_s)-(u_s(x),v_s(x))\|_2\geq\epsilon.
\end{equation}
Since $(u_s,v_s)$ and $(\tilde{u}_s,\tilde{v}_s)$ solve the same system of equations, subtraction corresponding equations, multiplying them by
$\tilde{u}_s-u_s$ and $\tilde{v}_s-v_s,$ and integrating over $\Omega,$ after simple manipulations we get
\[
(\mu C(\Omega) - C_s)\int |\tilde{u}_s-u_s|^2{}+{}|\tilde{v}_s-v_s|^2\,dx{}\leq{}0,
\]
where $C_s$ and $C(\Omega)$ as in \eqref{eq:stab}.
Since $\mu C(\Omega) - C_s$ is positive, we conclude that $\tilde{u}_s{}={}u_s$ and 
$\tilde{v}_s{}={}v_s.$ This clearly contradicts statement \eqref{eq:conv} and the lemma is proved.
\end{proof}



In the next theorem we show that the steady state $(u_s,v_s)$ is stable  and ``attracts"  solutions of the reaction-diffusion system, with the same source terms $f(x),$ $g(x),$ provided that the solution of the latter satisfy the stability condition \eqref{th:stab}.

\begin{lemma}
Let $(u,v)$ be a classical solution of the reaction-diffusion system \eqref{eq:u1}, \eqref{eq:v1} with initial data $u_0,v_0$ in $B(f,g,\mu)\cap C^{2+\alpha}$ and boundary conditions \eqref{bc:exist}. Then,  for any $t>0,$
\begin{equation}
\|(u(x,t)-u_s(x),v(x,t)-v_s(x))\|_2 \leq e^{-Kt} \|(u_0(x)-u_s(x),v_0(x)-v_s(x))\|_2,    
\end{equation}
where $K= 2(\mu C(\Omega)- C_s).$
\end{lemma}

\begin{proof}

Let $(U,V) = (u-u_s, v-v_s)$. 
Subtracting corresponding equation for $(u,v)$ and $(u_s,v_s)$ we get,
\begin{equation*}\partial_t{U} - \mu \Delta{U} = -aU + b \left(\frac{u^2}{1+v} - \frac{u_s^2}{1+v_s}\right),
\end{equation*}
\begin{equation*}\partial_t{V} - \mu \Delta{V} = -cV + d \left(u^2-u_s^2\right).
\end{equation*}
As in the proof of lemma \ref{lemma:ut} we obtain 
\begin{equation*}
    \frac{1}{2}\ \frac{d}{dt}\int |U|^2+|V|^2\ dx{}+{}
    2(\mu C(\Omega) - C_s)\int |U|^2{}+{}|V|^2\,dx{}\leq{}0.
\end{equation*}
Using Gronwall's inequality, we get that,
\begin{equation*}
\int |U(x,t)|^2+|V(x,t)|^2\ dx \leq e^{-Kt} \int |U_0(x)|^2+|V_0(x)|^2\ dx. \end{equation*}
where $K = 4(\mu C(\Omega) -C_s)>0.$
\end{proof}


\section{Conclusions}

In this paper we address the mathematical modeling of  recent experimental studies on  self-organization of human embryonic stem cells during early stages embryo's development.
Although several models based on reaction-diffusion equations were proposed in literature, those results are   only partially satisfactory as they either contain a number of artificial assumptions on the reaction part of the model or use initial and boundary conditions that do not correspond to the experimental setup. 

We showed that an Gierer-Meinhardt system of reaction-diffusion equations, with properly selected reaction coefficients and 
supplemented with Robin-type boundary conditions, qualitatively reproduces many of the experimental findings, thus identifying a proper mathematical framework. In this paper we only present numerical results for circular domains for brevity of presentation. Additional experiments with domains of irregular shapes, including non-convex domains, were presented  by Bedekar \cite{thesis}. All numerical simulations confirm very good qualitative agreement between our models and in vitro experiments.
Moreover, the model  produces  various new phenomena for the reaction-diffusion system under investigation such as an interesting instability investigated numerically in section \ref{sec:3.4}.

The second part of the paper is motivated by the numerical results obtained in the first part, and addresses the existence of of non-homogeneous steady state solutions and the asymptotically attract  solutions of the reactions-diffusion system. 
In general, this is a hard mathematical problem, which we were able to address under certain simplifying assumptions about the system.

With proper ramifications, the model considered in this paper can potentially lead to important scientific insights into the behavior of the biological system. In particular, instabilities outlined in numerical experiments performed here warrant
careful analytical investigation. In addition, we can use experimental data to estimate parameters in the PDE 
model via a Bayesian approach and use the resulting realistic model to predict outcomes of experiments in 
domains of various sizes. We intend to carry out further investigation of the reaction-diffusion model presented here
in subsequent papers.

\section{Appendix}
We use the following activator-inhibitor system for the dynamics of BMP/Nogin:
\begin{eqnarray*}
\partial_t u{}-{}\mu_{bmp}\Delta u{}&=&{}-\lambda_{bmp}u {}+{}k_{bmp}\frac{u^2}{\tilde{v}+v},\\
\partial_t v{}-{}\mu_{nog}\Delta v{}&=&{}-\lambda_{nog}v{}+{}\frac{k_{nog}}{\tilde{u}}u^2,
\end{eqnarray*} 
where $\tilde{u},$ $\tilde{v}$ are some reference values for BMP4 and Noggin. The boundary conditions are
\[
\frac{\partial u}{\partial n}{}={}H_{bmp}\left(\bar{u}-u\right),\quad 
\frac{\partial v}{\partial n}{}={}-H_{nog} v.
\]
where $\bar{u}$ is the background value of BMP4, and $H_{bmp},\, H_{nog}$ are positive numbers.
The initial conditions: $u(x,0){}={}\bar{b},$ $v(x,0)=0,$ which correspond to a cell colony being treated with high concentration of BMP4. The typical magnitudes of the parameters are listed in Table \ref{tab:param}.

\begin{table}
\label{tab:param}
\centering
\begin{tabular}{@{} lcc  @{}}    \toprule 
 {\rm Parameters} & {\rm Values} & {\rm SI\,\, units}\\ \midrule
    $\mu_{bmp}$ & 11 &  $\mu m^2/sec$\\
    $\mu_{nog}$ & 55 &  $ \mu m^2/sec$\\
    $\lambda_{bmp}$ & $9\times10^{-4}$ & $1/sec$ \\
    $k_{bmp}$ & $9\times10^{-4}$ & $1/sec$ \\
    $\lambda_{nog}$ & $9\times10^{-4}$ & $1/sec$ \\
    $k_{nog}$ & $9\times10^{-4}$ & $1/sec$\\ \bottomrule
    \\
\end{tabular}
\caption{Values for the diffusion and reaction parameters.}
\end{table}
The colony size (radius of the disk) $L= 500\,\mu m,$ and a typical experiment takes up to 3 days ($3\tau,$ $\tau{}={}86400sec$).
The experimental data on the values of $H_{bmp}$ and $H_{nog}$ are not available. We set them to $1\,(\mu m)^{-1}.$

Scaling the variables: $x\to Lx,$ $t\to \tau t,$ $u\to \tilde{u} u,$ $v\to \tilde{v}v,$ with $\tilde{u}/\tilde{v}=1,$ we obtain a system with non-dimensionless coefficients:
\begin{eqnarray*}
\partial_t u{}-{}\frac{\mu_{bmp}\tau}{L^2}\Delta u{}&=&{}-(\lambda_{bmp}\tau) u {}+{}(k_{bmp}\tau) \frac{u^2}{1+v},\\
\partial_t v{}-{}\frac{\mu_{nog}\tau}{L^2}\Delta v{}&=&{}-(\lambda_{nog})\tau v{}+{} (k_{nog}\tau) u^2,
\end{eqnarray*}
with the boundary conditions
\[
\frac{\partial u}{\partial n}{}={}H_{bmp}L\left(\frac{\bar{u}}{\tilde{u}}-u\right), \quad 
\frac{\partial v}{\partial n}{}={}-H_{nog}L v.
\]
This leads to the system \eqref{eq:u}, \eqref{eq:v} with coefficients $h_u = H_{bmp}L,$ $h_v = H_{nog} L,$ $\mu_u = \mu_{bmp}\tau L^{-2},$ $\mu_v = \mu_{nog}\tau L^{-2},$ $a = \lambda_{bmp}\tau,$ $b=k_{bmp}\tau,$ $c = \lambda_{nog}\tau,$ $d=k_{nog}\tau,$ the values of which are listed in Table \ref{tab:1}. For the boundary and initial conditions, ratio $\bar{u}/{\tilde{u}}=3.$ The scaling for Wnt/DKK and Nodal/Lefty RD systems are similar.

The numerical simulations are based on the forward Euler approximation of time derivatives with finite element methods, using piece-wise linear functions for the space discretization, and a suitable triangulation of the domain. Space and time partitions steps, $(h,\delta)$ were set to $h=10^{-3}$ and $\delta{}={}10^{-6},$ with $\delta= h^2,$ to prevent numerical instabilities.  The method was implemented using FreeFem++, see Heicht \cite{Hecht}.

\end{document}